\documentclass[a4paper,leqno,10pt]{article}

\usepackage[T1]{fontenc}                                               
\usepackage{mathtools,amssymb,amsthm}
\usepackage{mathrsfs}
\mathtoolsset{mathic}
\usepackage{lmodern}
\usepackage{eucal}
\usepackage[colorlinks,
    linkcolor={red!50!black},
    citecolor={blue!50!black},
    urlcolor={blue!80!black}]{hyperref}
\usepackage{microtype}
\usepackage{enumerate}
\usepackage{xcolor}
\usepackage[nameinlink,capitalise]{cleveref}
\usepackage[labelfont=bf]{caption}
\usepackage[textwidth=13cm,textheight=21cm]{geometry}

\theoremstyle{plain}
\newtheorem{thm}{Theorem}[subsection]
\newtheorem{lem}[thm]{Lemma}
\newtheorem{coro}[thm]{Corollary}
\newtheorem{thmalph}{Theorem}
\newtheorem{coralph}[thmalph]{Corollary}

\newtheorem{prop}[thm]{Proposition}

\newtheorem{theorem}[thm]{Theorem}

\newtheorem*{claim}{Claim}

\theoremstyle{definition}
\newtheorem{rmrk}[thm]{Remark}
\newtheorem{example}[thm]{Example}
\theoremstyle{remark}

\newcommand{\R}{\mathbb{R}}
\newcommand{\Z}{\mathbb{Z}}

\newcommand{\K}{\mathbb{K}}
\renewcommand{\H}{\mathbb{H}}
\newcommand{\C}{\mathbb{C}}

\newcommand{\Q}{\mathbb{Q}}

\newcommand*{\rom}[1]{\romannumeral}

\renewcommand{\Im}{\textup{Im}}
\providecommand{\msc}[1]{{\noindent\small\textbf{Mathematics Subject Classification (2020)} --- #1.}}
\providecommand{\keywords}[1]{{\noindent\small\textbf{Keywords} --- #1.}}
\newcommand{\rot}{\mathrm{rot}}
\newcommand{\geod}{\mathrm{geod}}
\newcommand{\bma}{\begin{pmatrix}}
\newcommand{\ema}{\end{pmatrix}}
\newcommand{\bigzero}{\mbox{\normalfont\large\bfseries 0}}
\newcommand{\bigId}{\mbox{\normalfont\large\bfseries \textrm{I}}}

\DeclareMathOperator{\SL}{SL}
\let\sl\relax
\DeclareMathOperator{\sl}{\mathfrak{sl}}
\DeclareMathOperator{\PSL}{PSL}
\DeclareMathOperator{\GL}{GL}

\DeclareMathOperator{\PSO}{PSO}

\DeclareMathOperator{\IM}{Im}
\DeclareMathOperator{\RE}{Re}
\DeclareMathOperator{\Tr}{Tr}
\DeclareMathOperator{\tr}{tr}
\DeclareMathOperator{\End}{End}
\DeclareMathOperator{\Vol}{Vol}

\DeclareMathOperator{\Id}{Id}
\DeclareMathOperator{\Ad}{Ad}

\DeclareMathOperator{\spec}{spec}

\DeclareMathOperator{\sign}{sign}
\DeclareMathOperator{\Res}{Res}

\DeclareMathOperator{\im}{Im}

\DeclareMathOperator{\Fix}{Fix}
\numberwithin{equation}{section}

\DeclareMathOperator{\tor}{tor}
\DeclareMathOperator{\mult}{mult}

\renewcommand{\thethm}{
\ifnum\value{subsection}=0
\thesection
\else
\thesubsection
	\fi
		.\arabic{thm}
}

\providecommand{\contact}{{
	\bigskip
	\small
	
	\noindent
	\textsc{L.~B\'{e}nard}:
	Institut de Math\'ematiques de Marseille, Aix--Marseille Universit\'e, 
	Site de Saint Charles, 
	3 place Victor Hugo, 
	Case 19, 
	13331 Marseille C\'edex 3,
	France
	\par\noindent\nopagebreak
	\textit{E-mail address:} \href{mailto:leo.benard@univ-amu.fr}
	{\texttt{leo.benard@univ-amu.fr}}\\

	\noindent
	\textsc{J.~Frahm}:
	Department of Mathematics, Aarhus University, Ny Munkegade 118, 8000 Aarhus C,
	Denmark
	\par\noindent\nopagebreak
	\textit{E-mail address:} \href{mailto:frahm@math.au.dk}
	{\texttt{frahm@math.au.dk}}\\
	
	\noindent
	\textsc{P.~Spilioti}:
	Mathematisches Institut, Georg--August Universit\"{a}t G\"{o}ttingen, Bunsenstra{\ss}e 3--5, 37073 G\"{o}ttingen, Germany
	\par\noindent\nopagebreak
	\textit{E-mail address:} \href{mailto:polyxeni.spilioti@mathematik.uni-goettingen.de}
	{\texttt{polyxeni.spilioti@mathematik.uni-goettingen.de}}
}}

\title{The twisted Ruelle zeta function on compact hyperbolic orbisurfaces and Reidemeister--Turaev torsion}
\author{L\'{e}o B\'{e}nard, Jan Frahm, Polyxeni Spilioti}
\date{November 17, 2023}

\begin{document}
\maketitle

\begin{abstract}
Let $X$ be a compact hyperbolic surface with finite order singularities, $X_1$ its unit tangent bundle. We consider the Ruelle zeta function $R(s;\rho)$ associated to a representation $\rho \colon \pi_1(X_1) \to \GL(V_\rho)$. If $\rho$ does not factor through $\pi_1(X)$, we show that the value at $0$ of the Ruelle zeta function equals the sign-refined Reidemeister--Turaev torsion of $(X_1, \rho)$ with respect to the Euler structure induced by the geodesic flow and to the natural homology orientation of $X_1$. It generalizes Fried's conjecture to non-unitary representations, and solves the phase and sign ambiguity in the unitary case. We also compute the vanishing order and the leading coefficient of the Ruelle zeta function at $s=0$ when $\rho$ factors through $\pi_1(X)$.
\end{abstract}

\bigskip
\keywords{hyperbolic orbisurface, twisted Ruelle zeta function, non-unitary representation, Reidemeister--Turaev torsion, Selberg trace formula}
\newline
\msc{Primary: 11M36; Secondary: 11F72, 37C30, 57Q10}

\section*{Introduction}

Given a compact Riemannian manifold $M$, a representation $\rho \colon \pi_1(M) \to \GL(V_\rho)$ and a vector field $\mathfrak X$ on $M$, the Ruelle zeta function is, under certain assumptions on $\rho$ and $\mathfrak X$, defined as
$$R(s;\rho) = \prod_{\gamma \textup{ prime}}\det\left(\Id-\rho(\gamma)e^{-s\ell(\gamma)}\right), $$
where $\gamma$ runs through the \emph{prime} periodic orbits of $\mathfrak X$ and $\ell(\gamma)$ denotes the length of the orbit $\gamma$.

In this work we are interested in the case where $M=X_1$ is the unit tangent bundle of a compact hyperbolic surface $X$ with finitely many singular points $x_1, \ldots, x_r$ of finite order. The generator $\mathfrak X$ of the geodesic flow acts on $X_1$ and its (prime) periodic orbits are lifts of (prime) closed geodesics on the surface.
For each singular point $x_j$ on $X$ there is a class of loops $c_j$ encircling $x_j$, which has finite order $\nu_j \in \Z_{>0}$ in the orbifold fundamental group of $X$. 
Their lifts (still denoted by $c_j$) to $\pi_1(X_1)$ satisfy $c_j^{\nu_j}= u$,  
where $u$ is the class of the generic fiber. The latter represents a loop in $X_1$ of a unit vector which makes a full positive rotation along the fiber around the base point.

Let $\rho \colon \pi_1(X_1) \to \GL(V_\rho)$ be a representation with $\dim V_\rho=n$. For any $j=1, \ldots, r$ we denote by $n_j= \dim \Fix  \rho(c_j)$, so that $\rho(c_j) = I_{n_j} \oplus T_j$.

The main result of this article is the computation of the behavior at zero of the Ruelle zeta function associated to $\rho$.

\begin{thmalph}[\cref{prop:torsion} and \cref{theo:Ruelle0}]\label{theo:mainFried}
	For any  irreducible representation $\rho\colon \pi_1(X_1) \to \GL(V_\rho)$, the Ruelle zeta function $R(s;\rho)$ converges on some right half plane in $\C$ and extends meromorphically to the whole complex plane. Moreover:
\begin{enumerate}
\item \label{i1}
If $\rho(u)= \Id_{V_\rho}$, then $R(s;\rho)$ vanishes at $s=0$ with order and leading coefficient prescribed by:
$$R\left(\frac{s}{2\pi}, \rho\right) \sim_{s \to 0}  \pm \frac{s^{n(2g-2+r)-\sum_{j=1}^rn_j}}{\prod_{j=1}^r  |\det(I_{n-n_j} -T_j)|\,(-\nu_j)^{-n_j}}.$$

\item \label{i2}
If $\rho(u) \neq\Id_{V_\rho}$, then the representation $\rho$ is acyclic. Let $\mathfrak e_{\geod}$ be the Euler structure induced by the geodesic flow on $X_1$. Then
$$R(0;\rho) = \pm\tor(X_1,\rho,\mathfrak {e}_{\geod}, \omega^1)$$
 where $\tor(X_1,\rho, \mathfrak e_{\geod}, \omega^1)\in\C^\times$ denotes the Reidemeister--Turaev torsion of $X_1$ in the representation $V_\rho$, the Euler structure $\mathfrak e_{\geod}$ and the natural homology orientation $\omega^1$ (see below for details).
\end{enumerate}
The sign is in both cases equal to $(-1)^{\mult(0;\Delta_{\tau,\rho}^\sharp)}$, where $\mult(0;\Delta_{\tau,\rho}^\sharp)$ is the multiplicity of the eigenvalue $0$ of the twisted Bochner--Laplacian $\Delta_{\tau,\rho}^\sharp$ acting on sections of a certain orbifold bundle over $X_1$ (see \cref{subsec:Laplacians} for details).
\end{thmalph}

When the representation $\rho$ satisfies $\rho(u) \neq\Id_{V_\rho}$ we prove that the twisted homology groups $H_*(X_1, V_\rho)$ are trivial (see \cref{sec:Reidemeister}) and one can define a combinatorial invariant, the Reidemeister--Turaev torsion. A theorem of Chapman (see \cite[Theorem 1]{Chapman}) states that it is indeed a topological invariant, although it depends in general on the choice of an Euler structure~(see \cite{Turaev_book}), that is, a choice of a lift of each cell of a cell decomposition of $X_1$ to its universal cover $\widetilde{X_1}$. To determine the sign of the torsion, one needs an additional data: a homology orientation. For the unit tangent bundle $X_1$, there is a natural choice of an Euler structure $\mathfrak e_{\geod}$ induced by the geodesic flow, and a natural homology orientation $\omega^1$ (see \cref{subsec:Euler}). We compute explicitly in \cref{subsec:Torsion} the Reidemeister--Turaev torsion $\tor(X_1, \rho, \mathfrak e_{\geod}, \omega^1) \in \C$ for this Euler structure and homology orientation. \cref{theo:mainFried}~\cref{i2} says that it coincides with the value at zero of the Ruelle zeta function.

In the case where the representation $\rho$ is unitary, the absolute value of the Reidemeister torsion is well-defined, independently from additional datas. In this setting \cref{theo:mainFried} generalizes \cite[Theorem 2 and 3]{fried1986fuchsian}, where Fried proved a similar statement for unitary and \emph{classical} representation. He claimed without a proof that his results still hold true for general (non-classical) unitary representation. Note that in this paper we use the inverse convention for the torsion, so that our $\tor(X_1, \rho)$ is Fried's $\tau_\rho(X_1)^{-1}$, see \cref{rmk:convention}.

Our work fills the previous gap and extends Fried's results to non-unitary representations. For $\rho(u)=\Id_{V_\rho}$, we also determine the sign of the coefficient governing the vanishing of the Ruelle zeta function at zero. For $\rho(u) \neq \Id_{V_\rho}$, even in the unitary case, we drop the phase and sign ambiguity for the torsion using the accurate Euler structure and homology orientation. We show that it equals the value of the Ruelle zeta function at zero, up to a sign which we explicitly compute.

Note that the value $s=0$ is not in the convergence domain of the infinite product $R(s; \rho)$, but as we mentioned above the Ruelle zeta function extends meromorphically. On the other hand, this function strongly depends on the metric we fixed on the orbifold $X$ and on the vector field acting on $X_1$. A striking consequence of \cref{theo:mainFried}~\cref{i2} is that, for a given hyperbolic metric on $X$ and for the associated geodesic flow,  the value at zero of the Ruelle zeta function is independent of these choices.

In the case the surface has no elliptic points, \cref{theo:mainFried}~\cref{i1} recovers \cite[Corollary C]{FS}. Note that \cref{theo:mainFried}~\cref{i2} implies that both the Ruelle at $0$ and the torsion are
locally constant on the representation variety, and invariant under the action of the mapping class group.

\paragraph{Relation to previous results.}
In a more general setting, metric independence of the value of Ruelle zeta functions at $0$, and its relation with Reidemeister torsion, has been widely studied in the past almost 40 years. For $\rho$ unitary, it is known as \textit{Fried's conjecture} and has been settled as a theorem for compact, hyperbolic odd-dimensional manifolds by Fried in \cite{Fried86}, using the Selberg trace formula. Bunke and Olbrich in~\cite{BO} unified the treatment for the twisted dynamical zeta functions with unitary representations, for all locally symmetric spaces of real rank one, using the Selberg trace formula and tools for harmonic analysis on locally symmetric spaces. The higher rank case was treated by Moscovici and Stanton in \cite{Moscovici_Stanton}, and Shen in \cite{Shen18}. A similar result, but in variable curvature for $3$-dimensional manifolds, has been proved by Dang--Guillarmou--Rivi\`ere--Shen in \cite{dang2019fried}, using completely different techniques. 

In the untwisted case (without representation $\rho$), Fried showed in \cite{fried1986fuchsian} that the vanishing order of the Ruelle zeta function at zero was given by the Euler characteristic of the surface. Hence he recovered the fact that the length spectrum of a hyperbolic surface determines its topology, which can be obtained as an easy consequence of the Selberg trace formula. Notably, this has recently been extended using microlocal analysis methods to the case of variable negative curvature by Dyatlov--Zworski in \cite{dyatlov2017ruelle}. On the other hand,  Cekic--Delarue--Dyatlov--Patternain \cite{CDDP22}  studied the value of the untwisted Ruelle zeta function at zero for conformal metric deformations of hyperbolic 3-manifolds. They showed that even under small perturbations the vanishing order will jump.

In the case of non-unitary representations $\rho$, less is currently known, despite that it is what occurs in many interesting cases, such as holonomy representations of hyperbolic manifolds. If $\rho$ is obtained as the restriction of a representation of a Lie group $G$, then Fried's conjecture has been established for $G=\SL_2(\C)$ by M\"uller~\cite{M2} (using the PhD work of Wotzke \cite{Wo}) and generalized to arbitrary $G$ such that $G/K$ is odd-dimensional by Shen~\cite{Shen21}. Moreover, M\"uller (\cite{muller2021ruelle}), Shen (\cite{Shen2020}) and the third author~(\cite{Spilioti20}) proved Fried's conjecture for more general non-unitary representations, with the additional assumption that they are close to a unitary and acyclic representation in the representation variety. In a completely general setting, Chaubet--Dang established in \cite{chaubet2019dynamical} a variational formula relating some dynamical torsion to the Reidemeister--Turaev torsion. Under some analytic hypothesis on the vector field $\mathfrak X$, for acyclic $\rho$, this dynamical torsion coincides with the value at zero of the Ruelle zeta function under study here. Namely, the operator given by the $\rho$-twisted Lie derivative of $\mathfrak X$ should not have $0$ as a Ruelle resonance. As was pointed out to us by the referee, this is known to be true in the case $\rho$ is unitary, for any flow close enough to the geodesic flow of a hyperbolic metric, see \cite[Lemma 7.4]{dang2019fried}.

 As a consequence of their work together with our \cref{theo:mainFried} we have the following corollary:
\begin{coralph}
The statements in \cref{theo:mainFried} still hold true for $(\rho, \mathfrak X)$ in an open neighborhood of $(\rho_0, \mathfrak X_\geod)$ in $\operatorname{Hom}(\pi_1(X_1), \GL(V_{\rho_0})) \times C^{\infty}(X_1, TX_1)$ , for $\rho_0 \colon \pi_1(X_1) \to \GL(V_{\rho_0})$ a unitary acyclic representation and $\mathfrak X_\geod$ the geodesic flow of any hyperbolic metric on X.
\end{coralph}

\begin{proof}
Starting with $\rho_0 \colon \pi_1(X_1) \to \GL(V_{\rho_0})$ a unitary acyclic representation and $\mathfrak X_\geod$ the geodesic flow of any hyperbolic metric on $X$, since $0$ is not a Ruelle resonance, using \cite[Theorem 2]{dang2019fried} one can vary the vector field $\mathfrak X$ in a neighborhood of $\mathfrak X_0$, such that the identity $R_{\mathfrak X}(0, \rho_0) = \pm \tor(X_1, \rho_0, \mathfrak e_{\mathfrak X})$
remains true. Still, $0$ is not a Ruelle resonance, and then using \cite[Theorem 4]{chaubet2019dynamical} we can vary the representation $\rho$ in a neighborhood of $\rho_0$.
\end{proof}

In the present paper, the representation $\rho$ is neither required to be unitary, nor does it have to be close to a unitary acyclic representation. We even allow representations for which the non-self adjoint Bochner--Laplacian might have some generalized $0$-eigenvalues.
This is because we do not use the relation between the Selberg zeta function and the refined analytic torsion. Indeed, we compute directly the value of the Ruelle zeta function at zero, and identify it with the Reidemeister torsion.

Recently, Yamaguchi~\cite{yamaguchi2020dynamical} was able to show \cref{theo:mainFried}~\cref{i2} in the special case where $\rho=\rho_{2N}$ is the restriction of the irreducible $2N$-dimensional representation of $\SL(2,\R)$ to $\pi_1(X_1)$. 
His proof uses the relation
$$ R(s;\rho_{2N}) = \frac{Z(s-N+\frac{1}{2})}{Z(s+N+\frac{1}{2})} $$
between the twisted Ruelle zeta function $R(s;\rho_{2N})$ and the non-twisted Selberg zeta function $Z(s)$ (associated with the trivial representation of $\pi_1(X_1)$) and hence does not generalize to arbitrary representations $\rho$ of $\pi_1(X_1)$.

\medbreak

\paragraph{Outline of the proof.}

Our proof uses the intimate relation between the twisted Ruelle zeta function $R(s;\rho)$ and the twisted Selberg zeta function $Z(s;\rho)$ (see \cref{subsec:mero}):
$$ R(s;\rho) = \frac{Z(s;\rho)}{Z(s+1;\rho)}. $$
The behavior of $Z(s;\rho)$ at $s=0$ and $s=1$ can be studied through its functional equation relating $Z(s;\rho)$ and $Z(1-s;\rho)$ which we derive in \cref{theo:Funct}. Both the meromorphic continuation and the functional equation of $Z(s;\rho)$ are obtained from a twisted Selberg trace formula (see \cref{thm:TraceFormula}). This trace formula arises from the trace of the heat operator of M\"{u}ller's twisted Bochner--Laplacian $\Delta_{\tau,\rho}^\sharp$ acting on sections of a certain vector bundle $E_{\tau,\rho}$ over $X$ (see M\"{u}ller~\cite[Section 4]{M1} for the definition of the twisted Bochner--Laplace operator). The vector bundle $E_{\tau,\rho}$ over $X$ is associated with the representation $\rho$ of $\pi_1(X_1)$ and a character $\tau$ of the universal covering group of $\PSO(2)$. The twist by the character $\tau$ is necessary in order to obtain a vector bundle over~$X$ (see \cref{subsec:vect} for details). We note that $\tau$ can only be trivial in the case where $\rho$ factors through $\pi_1(X)$ and hence this twisted construction is crucial for \cref{theo:mainFried}~\cref{i2}.

One of the technical parts of the proof consists of a detailed analysis of the identity, hyperbolic and elliptic contribution to the geometric side of the trace formula. This generalizes Hejhal's trace formula~\cite[Chapter 9, Theorem 6.2]{He2} to the case of non-unitary representations $\rho$ and uses Hoffmann's computations of orbital integrals for the universal covering group of $\PSL(2,\R)$. The main difference is the use of the non-self adjoint operator $\Delta_{\tau,\rho}^\sharp$. This operator is not necessarily diagonalizable and we have to work with generalized eigenspaces. Moreover, the eigenvalues of this operator are not necessarily positive real numbers but complex (see \cref{subsec:Laplacians} for further details).
We remark that the generalized eigenfunctions of $\Delta_{\tau,\rho}^\sharp$ can be viewed as automorphic forms on $\H^2$ of a certain weight with values in the non-unitary representation $\rho$ (see \cite[Chapter 9]{He2}, and \cite{Hof94} for the unitary case).

On the torsion side, the simple combinatorial nature of the unit tangent bundle allows to compute explicitly the Euler structure $\mathfrak e_{\geod}$ induced by the geodesic flow, and the Reidemeister--Turaev torsion in this Euler structure. It relies on Turaev's correspondence between \emph{combinatorial} Euler structures, which yield a choice of lifts of a cell decomposition of $X_1$ to its universal cover necessary to compute the torsion, and \emph{smooth} Euler structures, given by non-vanishing vector fields on $X_1$.
To our best knowledge, it is the only case this computation has been performed explicitly, since this correspondence is not very explicit in the direction we use. Namely, given a non-vanishing vector field on a 3-manifold $M$, it looks in general like a difficult task to find the corresponding combinatorial Euler structure.

Finally, the last part of the proof consists of an explicit computation of the Ruelle zeta function at $s=0$. It follows the ideas of \cite[Section 3]{fried1986fuchsian}, where Fried computed the modulus of each term contributing to $R(0;\rho)$ from the functional equation of the Selberg zeta function, in the case $\rho$ was unitary. In his case the Ruelle zeta function is real for $s$ real, hence he argued that he could forget the arguments of the terms occurring in the computation. Of course, in our case where $\rho$ is no longer unitary, these are exactly what we have to compute, and it turns out that after tedious computations (\cref{subsec:Fried}) they kind of miraculously cancel out with each other.

Let us finally comment that some arguments of the proofs (such as the trace formula (\cref{thm:TraceFormula}) and the functional equation (\cref{theo:Funct})) might be deduced from an analytic continuation argument from the known unitary case, if one knew that the unitary representations were Zarisky dense in the representation variety of the unit tangent bundle $X_1$. This kind of argument has been already used by several authors in various contexts, see Anantharaman (\cite{Anan}) and Braverman--Kappeler (\cite{BK}).

\medbreak
\paragraph{Organization of the paper.}
The paper is organized as follows.
In \cref{sec:Prel}, we introduce our geometric setting, recall some facts from the representation theory
of $\widetilde{\PSL(2,\R)}$ and fix notation.
In \cref{sec:Reidemeister}, we define and compute the Reidemeister--Turaev torsion of the unit tangent bundle $X_1$ of the hyperbolic orbifold $X$ in the Euler structure induced by the geodesic flow. In \cref{sec:STF}, we establish the Selberg trace formula for the twisted, non-self adjoint Bochner--Laplacian on $X$,
which will be the key point to prove \cref{theo:mainFried}.
In \cref{sec:Ruelle}, we introduce the twisted Ruelle and Selberg zeta functions 
and prove their meromorphic continuation to the whole complex plane.
Finally, we compute the value of the Ruelle zeta function at zero and
prove \cref{theo:mainFried}.

\paragraph{Acknowledgments.} The authors would like to thank Adrien Boulanger, Fran\c{c}ois Costantino, Nguyen Viet Dang, Pierre Dehornoy, Werner M\"{u}ller, Joan Porti and Louis-Hadrien Robert for helpful discussions and comments. We also thank the anonymous referees for their comments which led to a significant improvement of the paper. The first and third author are partially funded by the Research Training Group 2491 ``Fourier Analysis and Spectral Theory'', University of G\"{o}ttingen. The second and third author were partially supported by a research grant from the Villum Foundation (Grant No. 00025373).

\section{Preliminaries}
\label{sec:Prel}
We collect some basic information about compact hyperbolic orbisurfaces $X$ (\cref{subsec:Orbi}), their unit tangent bundles $X_1$ and their fundamental groups $\pi_1(X_1)$ (\cref{subsec:Unit}), and recall some representation theory of the universal cover of $\PSL(2,\R)$ (\cref{subsec:RepTheory}).

\subsection{Compact hyperbolic orbisurfaces}
\label{subsec:Orbi}
We consider the upper half plane
\[
\H^2:=\{z=x+iy:y>0\}.
\]
The group $G=\PSL(2,\R)$ acts on $\H^2$ by fractional linear transformations. This action is transitive and the stabilizer of $i\in\H^2$ is the maximal compact subgroup $K=\PSO(2)$, hence
\begin{equation*}
\H^2\cong G/K.
\end{equation*}
Let $\mathfrak{g}=~\mathfrak{sl}(2,\R)$ be the Lie algebra of $G$ and let $\mathfrak{g}=\mathfrak{k}\oplus\mathfrak{p}$ be the Cartan decomposition of $\mathfrak{g}$ with respect to the maximal compact subgroup $K$.
The restriction of the Killing form on $\mathfrak{g}$ to $\mathfrak{p}$ induces a $G$-invariant metric on $\H^2$ which is a multiple of the Poincar\'{e} metric
\begin{equation*}
ds^2=\frac{dx^2+dy^2}{y^2}.
\end{equation*}
Let $\Gamma\subseteq G$ be a cocompact Fuchsian group, i.e., a discrete subgroup of $G$ such that the quotient $X=\Gamma\backslash \H^2=\Gamma\backslash G/K$ is compact. Then, the Poincar\'{e} metric on~$\H^2$ induces a metric on $X$ that turns it into a compact hyperbolic orbisurface. Note that in this case, every non-trivial element in $\Gamma$ is either hyperbolic or elliptic.

Let $\widetilde{G}$ denote the universal cover of $G$ and write~$\widetilde{H}$ for the preimage of a subgroup $H\subseteq G$ under the universal covering map. Consider the following one-parameter families in $G$ (modulo $\pm I_2$):
$$ k_\theta=\begin{pmatrix}\cos\theta&\sin\theta\\-\sin\theta&\cos\theta\end{pmatrix}, \qquad a_t=\begin{pmatrix}e^{\frac{t}{2}}&0\\0&e^{-\frac{t}{2}}\end{pmatrix}, \qquad n_x=\begin{pmatrix}1&x\\0&1\end{pmatrix}, $$
where $\theta,t,x\in\R$. Denote by $\widetilde{k}_\theta$, $\widetilde{a}_t$ and $\widetilde{n}_x$ the unique lifts to the universal cover $\widetilde{G}$, turning
$$ \widetilde{K}=\{\widetilde{k}_\theta:\theta\in\R\}, \qquad \widetilde{A}=\{\widetilde{a}_t:t\in\R\} \qquad \mbox{and} \qquad \widetilde{N}=\{\widetilde{n}_x:x\in\R\} $$
into connected one-parameter subgroups of $\widetilde{G}$. Then the Iwasawa decomposition
$$ \widetilde{G} = \widetilde{N}\widetilde{A}\widetilde{K} $$
holds. Note that the center $\widetilde{Z}$ of $\widetilde{G}$ is given by
$$ \widetilde{Z} = \{\widetilde{k}_\theta:\theta\in\Z\pi\}. $$

Finally, we define two vector fields on $X_1 = \Gamma \backslash G$:
\begin{enumerate}
\item The (generator of the) geodesic flow $\mathfrak X$, which is induced by the infinitesimal action of the group $(a_t)_{t \in \R}$.
\item The rotation in the fiber vector field $\mathfrak X_\rot$, generated by the infinitesimal action of the group $(k_\theta)_{\theta \in \R}$.
\end{enumerate}
\subsection{The relation between $\pi_1(X_1)$ and $\pi_1(X)$}
\label{subsec:Unit}
The aim of this section is to shed light on the relations between the fundamental group of the orbifold surface $X$ and the fundamental group of its unit tangent bundle $X_1$. For this, we follow \cite[Section 2]{fried1986fuchsian}.

\medbreak

We first note that the unit tangent bundle of $\H^2$ naturally identifies with $G=\PSL(2,\R)$. Indeed the map
\begin{align*}
G &\to \H^2\\
g &\mapsto g\cdot \textrm i
\end{align*}
is a fibration, whose fiber is a circle $S^1$ of directions of unit tangent vectors.

It follows from the preceding discussion that the quotient $\Gamma \backslash G$ is a compact 3-manifold $X_1$, which we will (somewhat abusively) call the unit tangent bundle of $X$. It is a Seifert fibered 3-dimensional manifold, and this fibration induces an exact sequence 
$$1 \to \Z = \pi_1(\PSO(2)) \to \pi_1(X_1) \to \Gamma= \pi_1(X) \to 1$$
where $\pi_1(X)$ denotes the orbifold fundamental group of $X$.

In fact, a compact orbisurface $X$ must have finitely many conical points $x_1, \ldots, x_r$. Removing small disc neighborhoods $D_1, \ldots, D_r$ of those points, one gets a compact surface $\overline X$ with boundary $\partial \overline X=\{\partial D_1, \ldots, \partial D_r\}$ a union of disjoints circles, and with fundamental group
$$\pi_1(\overline X)= \langle a_1, b_1, \ldots, a_g, b_g, c_1, \ldots, c_r \mid \prod_i[a_i,b_i]\prod_j c_j = 1 \rangle.$$
Capping off the boundary components $\partial D_i$ by orbidiscs, one gets the following presentation for the fundamental group of the orbisurface:
$$\pi_1(X)= \langle a_1, b_1, \ldots, a_g, b_g, c_1, \ldots, c_r \mid \prod_i[a_i,b_i]\prod_j c_j = 1, \, c_j^{\nu_j}=1 \rangle,$$
for some non-zero natural integers $\nu_j, \, j= 1, \ldots, r$. Those integers are the order of the elliptic conjugacy classes in $\Gamma$ corresponding to the singular points $x_j$.

To compute the fundamental group of the unitary tangent bundle $X_1$ of $X$, let us denote by $\langle u\rangle$ the fundamental group of the fiber $\PSO(2)$ which identifies with the center of $\widetilde{G}$. The generator $u=\widetilde{k}_\pi$ identifies with a loop in $X_1$ corresponding to a complete clockwise rotation of unit vectors around the base point in $X$. It is proved in \cite[Section 2]{fried1986fuchsian} that one gets the presentation
\begin{multline}\label{presentation}
\pi_1(X_1) = \langle a_1, b_1, \ldots, a_g, b_g, c_1, \ldots, c_r, u \mid \\ [u,a_i] =1, [u,b_i]=1, [u,c_j]=1, \,  \prod_i[a_i,b_i]\prod_j c_j = u^{2g-2+r}, \, c_j^{\nu_j}=u \rangle.
\end{multline}

Now, let $\rho:\pi_1(X_1)\to\GL(V_\rho)$ be a finite-dimensional complex representation of $\pi_1(X_1)$. Since $u\in\pi_1(X_1)$ is central, $\rho(u)$ commutes with $\rho(\gamma)$ for all $\gamma\in\pi_1(X_1)$. If $\rho$ is irreducible, $\rho(u)$ must be a scalar multiple of the identity by Schur's Lemma, say $\rho(u)=\lambda\,I_n$. We show that $\lambda$ is in fact a root of unity. For this, let $N = \operatorname{lcm}(1, \nu_1, \ldots, \nu_r)$ be the least common multiple of $1$ and the orders of the elliptic conjugacy classes. By definition, the orbifold Euler characteristic $\chi(X)$ of $X$ is given by
$$ \chi(X) =  2-2g+\sum_{j=1}^r\left(\frac 1 {\nu_j}-1\right) \in \Q. $$

\begin{lem}\label{lem:CenterActionIrredPi1X1}
Let $\rho\colon \pi_1(X_1) \to \GL(V)$ be an irreducible representation of dimension $n$. Then, $\rho(u) =\lambda\, I_n$ with $\lambda^{Nn\chi(X)}=1$
\end{lem}

\begin{proof}
Since $c_j^{\nu_j}=u$ we find
$\rho(c_j)^{\nu_j}= \lambda \, I_n$, in particular
$$\prod_{j=1}^r \rho(c_j)^N = \prod_{j=1}^r \lambda ^{N/\nu_j} \, I_n.$$
On the other hand, 
$$\prod_{j=1}^r \rho(c_j) = \lambda^{2g-2+r} \left(\prod_{i=1}^g[\rho(a_i),\rho(b_i)]\right)^{-1}.$$
Since $\det(\prod_{i=1}^g[\rho(a_i),\rho(b_i)]) = 1$, we deduce 
$\lambda^{Nn \left(\sum_{j=1}^r 1/\nu_j - (2g-2+r)\right)} = 1$
and the lemma follows.
\end{proof}

\subsection{Representation theory of $\widetilde{\PSL(2,\R)}$}
\label{subsec:RepTheory}
We briefly recall the principal series and the (relative) discrete series of the universal covering group of $\PSL(2,\R)$,  following \cite[Section 1]{Hof94}.

Let $\widetilde{M}=\widetilde{Z}$ denote the center of $\widetilde{G}$, then $\widetilde{N}\widetilde{A}\widetilde{M}$ is a parabolic subgroup of $\widetilde{G}$. The unitary dual of $\widetilde{K}\simeq\R$ is comprised of the unitary characters $\tau_m$, $m\in\R$, defined by
\begin{equation}
	\tau_m(\widetilde{k}_\theta) = e^{im\theta} \qquad (\theta\in\R).\label{eq:DefTauM}
\end{equation}
The restriction $\sigma_\varepsilon$ of $\tau_m$ to $\widetilde{M}$ only depends on $\varepsilon=m+2\Z\in\R/2\Z$, and the unitary dual of $\widetilde{M}$ is given by all $\sigma_\varepsilon$, $\varepsilon\in\R/2\Z$.

For $\sigma=\sigma_\varepsilon$, $\varepsilon\in\R/2\Z$, and $s\in\C$ we form the principal series representation $\pi_{\sigma,s}$ of $\widetilde{\PSL(2,\R)}$ on
$$ I_{\sigma,s} = \{f\in C^\infty(\widetilde{G}):f(na_tmg)=\sigma(m)e^{st}f(g)\}, $$
acting by right translation. The $\widetilde{K}$-types in $I_{\sigma,s}$ are spanned by the functions $\phi_m$ with $m+2\Z=\varepsilon$, where
$$ \phi_m(na_tk) = e^{st}\tau_m(k) \qquad (n\in\widetilde{N},t\in\R,k\in\widetilde{K}). $$

Note that the Casimir element $\Omega=\frac{1}{4}(H^2+2EF+2FE)$ with
$$ H=\begin{pmatrix}1&0\\0&-1\end{pmatrix}, \qquad E=\begin{pmatrix}0&1\\0&0\end{pmatrix}, \qquad F=\begin{pmatrix}0&0\\1&0\end{pmatrix} $$
acts in $\pi_{\sigma,s}$ by $s(s-1)\Id$.

\paragraph{Unitary principal series.} For $s=\frac{1}{2}+i\lambda$, $\lambda\in\R$, the representation $\pi_{\sigma,s}$ extends to a unitary representation on the Hilbert space
$$ \left\{f:\widetilde{G}\to\C:f(na_tmg)=\sigma(m)e^{(i\lambda+\frac{1}{2})t}f(g),\int_{\widetilde{M}\backslash\widetilde{K}}|f(k)|^2\,dk<\infty\right\}, $$
the \emph{unitary principal series}. We denote the distribution character of this representation by $\Theta_{\sigma,\lambda}$.

\paragraph{(Relative) discrete series.} For $m>0$, the $\widetilde{K}$-types $\phi_{m'}$, $m'\in\pm\{m,m+2,m+4,\ldots\}$, span an $\sl(2,\C)$-invariant subspace of $I_{\sigma_\varepsilon,\frac{m}{2}}$ where $\varepsilon=\pm m+2\Z\in\R/2\Z$, which can be completed to a Hilbert space of a unitary representation of $\widetilde{G}$, the \emph{(relative) discrete series}. We write $\Theta_{\pm m}$ for the distribution character of this representation. Details about the invariant inner product on the (relative) discrete series can be found in \cite[Section 1]{Hof94}.

\section{Reidemeister torsion}
\label{sec:Reidemeister}
In this section, we compute the Reidemeister--Turaev torsion of the unit tangent bundle~$X_1$ of a hyperbolic orbisurface $X$ for the Euler structure given by the geodesic flow. In \cref{subsec:Defi}, we define  the Reidemeister--Turaev torsion. In \cref{subsec:Euler}, we describe the two equivalent notions of Euler structures introduced by Turaev, we introduce the sign-refinement of the torsion induced by a homology orientation, and we discuss thoroughly the example of the solid torus.  We compute explicitly the Reidemeister-Turaev torsion of $X_1$ in the Euler structure given by the geodesic flow in \cref{subsec:Torsion}, and determine its sign.

\subsection{Twisted homology, Reidemeister--Turaev torsion and Euler structures}
\label{subsec:Defi}
In this section, we define the twisted homology and Reidemeister torsion of a finite CW-complex $W$ with a finite dimensional representation $\rho \colon \pi_1(W)  \to \GL(V_\rho)$ and an Euler structure $\mathfrak e$. References include \cite{Nicolaescu,Turaev,Turaev_book}.

\subsubsection{Twisted homology}

We will denote the fundamental group $\pi_1(W)$ by $\pi$. Let $\widetilde W$ be the universal cover of $W$. It has the structure of an (infinite) CW-complex, whose cells can be listed as follows:
if $\{c^i_1, \ldots, c^i_{k_i}\}$ are the $i$-dimensional cells of $W$, then $\{\pi\cdot \widetilde c^i_1, \ldots,\pi \cdot \widetilde c^i_{k_i}\}$ is the list of the $i$-dimensional cells of $\widetilde W$, where $\pi \cdot \widetilde c^i_j$ denotes the set $\{\gamma \cdot \widetilde c^i_j \mid \gamma \in \pi\}$, see an example in \cref{circle}.

\begin{figure}[h]
\begin{center}
\def\svgwidth{0.7\columnwidth}
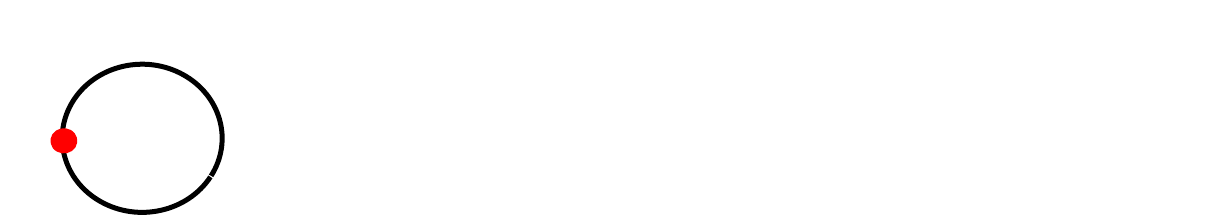
\caption{\label{circle} In this example, the CW complex $W$ is a circle, with one $0$-cell $c^0$ and one $1$-cell $c^1$. Its universal cover $\widetilde W$ is the real line, with the action of the fundamental group $\pi_1(W) = \langle u \rangle$. }
\end{center}
\end{figure}
It endows $C_*(\widetilde W)$ with the structure of a complex of $\mathbb Z[\pi]$-modules of finite rank, with a basis given by a choice of a lift $\widetilde c^i_j$ for each cell $c^i_j \in C_*(W)$.
Define the complex $$C_*(W, V_\rho) = V_\rho \otimes_{\mathbb Z[\pi]} C_*(\widetilde W).$$ Given $\{v_1, \ldots, v_n\}$ a basis of $V_\rho$, the set $\{ v_1 \otimes \widetilde c^i_1, v_2 \otimes \widetilde c^i_1 \ldots, v_{n-1} \otimes \widetilde c^i_{k_i}, v_n \otimes \widetilde c^i_{k_i}\}$ gives a basis of the vector space $C_i(W, V_\rho)$ for each $i$, and the boundary operator is given by $\textrm{id} \otimes \partial$.

\begin{example}
\label{ex:circle}
In the example of \cref{circle}, given a representation $\rho\colon \pi_1(\mathbb S^1) \to \GL(V_\rho)$, the complex is just
$$C_1(\mathbb S^1, V_\rho) \simeq V_\rho \to V_\rho \simeq C_0(\mathbb S^1, V_\rho)$$
and the boundary map is $\partial (v \otimes \widetilde c^1) = v \otimes (u-1)\cdot \widetilde c^0 = (\rho(u)-I_n) v \otimes \widetilde c^0$. In the case $\det(\rho(u)-I_n) \neq 0$, the boundary map $\partial \colon C_1(\mathbb S^1, V_\rho) \to C_0(\mathbb S^1, V_\rho)$ is an isomorphism. The twisted homology vector spaces $H_*(\mathbb S^1, V_\rho)$ are then trivial.
\end{example}

\begin{rmrk}
In  \cref{ex:circle},  the identification $C_i(\mathbb S^1, V_\rho) \simeq V_\rho$ is not canonical, it depends on the choice of a lift of the cell $c^i$. Different choices would modify the boundary map $\partial$ by a factor $\rho(u^m)$, for some $m \in \Z$. On the other hand, the homology of this complex does not depend on the choices.
\end{rmrk}

For a complex $C_*$, with boundary maps $\partial_*$, we denote by $Z_i = \ker \partial_i$, by $B_i = \im \partial_{i+1}$. The homology vector spaces of the complex are the spaces $H_i = Z_i/ B_i$. When all the spaces $H_i$ are trivial, as in \cref{ex:circle}, we say that the complex $C_*$ is \emph{acyclic}.

\subsubsection{Reidemeister torsion of an acyclic based complex}

Given an acyclic based complex $(C_*, \bf{c}^*)$ of vector spaces over a field $\K$, where for each $i$, the set $\mathbf{c}^i$ is a basis of the vector space $C_i$, the Reidemeister torsion is defined as follows.
The boundary maps yield exact sequences
$$ 0 \to Z_i \to C_i \xrightarrow{\partial_i} B_{i-1} \to 0.$$
Because the complex is acyclic, one has $B_i = Z_i$ for each $i$. If one picks arbitrary bases $\mathbf{b}^i$ of the spaces $B_i$ and arbitrary lifts $\mathbf{\overline b}^{i-1}$ in $C_{i}$ satisfying $\partial_i(\mathbf{\overline b}^{i-1}) =\mathbf{b}^{i-1}$, one obtains new bases $\mathbf{b}^i \sqcup \mathbf{\overline b}^{i-1}$ of $C_i$ for each $i$.\\
Let $\mathbf{b}^i \sqcup \mathbf{\overline b}^{i-1} \colon \mathbf{c}^i$ be the matrix which maps the basis $\mathbf c^i$ onto the basis $\mathbf{b}^i \sqcup \mathbf{\overline b}^{i-1}$.
We denote by $[\mathbf{b}^i \sqcup \mathbf{\overline b}^{i-1} \colon \mathbf{c}^i]$ the determinant of the change of basis matrix $\mathbf{b}^i \sqcup \mathbf{\overline b}^{i-1} \colon \mathbf{c}^i$, the Reidemeister torsion is defined as the alternating product
$$\tor(C_*, \mathbf{c}^*) = \prod_i [\mathbf{b}^i \sqcup \mathbf{\overline b}^{i-1} \colon \mathbf{c}^i]^{(-1)^i} \in \K^* / \{\pm 1\}. $$
It is defined up to sign, it does depend on the basis $\mathbf{c}^*$, but not on the choices~$\mathbf{b}^*$ nor on the lifts $\mathbf{\overline b}^*$.

\begin{rmrk}
\label{rmk:convention}
There is no standard convention regarding the choice of the torsion or its inverse. Namely, some authors use as a definition for the torsion the alternating product 
$$\prod_i [\mathbf{b}^i \sqcup \mathbf{\overline b}^{i-1} \colon \mathbf{c}^i]^{(-1)^{i+1}}.$$
At the end one arrives with the inverse of the quantity we defined here. It is the case in \cite{fried1986fuchsian}, for instance.
\end{rmrk}
\subsubsection{Reidemeister torsion of a CW complex}

In the case of a CW-complex $W$ with a representation $\rho \colon \pi_1(W) \to \GL(V_\rho)$, if the complex $C_*(W, V_\rho)$ is acyclic, we will say that $\rho$ is acyclic. One can compute the Reidemeister torsion as follows.\\
First, one needs to make a choice of lift for each cell of $W$ in $\widetilde W$. Then take an arbitrary basis of the vector space $V_\rho$, all this together yields a basis $\mathbf c^\ast$ of the complex $C_*(W, V_\rho)$. Then the Reidemeister torsion
$$\tor(W, \rho, \mathbf c^*)$$ 
depends on the choice of the lifts, as emphasized in the notation. This dependance restricts in a dependance of an \emph{Euler structure}, which we will explain now.

\subsection{Combinatorial and smooth Euler structures}
\label{subsec:Euler}
\subsubsection{Combinatorial Euler structure}
In \cite{Turaev}, Turaev investigated the concept of Euler structure. The main achievement is to show that Euler structures can be equivalently defined in two different ways.

An Euler chain in a CW complex $W$ is a singular one-chain $c$ in $W$, whose boundary is of the form
$$\partial c = \sum_{\sigma \in W} (-1)^{\dim \sigma} [w_\sigma],$$
where $w_\sigma$ is a point in $\sigma$. Note that Euler chains exist if and only if $\chi(W) = 0$, since the singular $0$-chain $\sum_{\sigma \in W} (-1)^{\dim \sigma} [w_\sigma]$ is then in the kernel of the augmentation map.

Take $c'$ another Euler chain with $\partial c'= \sum_\sigma (-1)^{\dim\sigma}[w_\sigma']$, and some arbitrary paths $\gamma_\sigma$ in each cell $\sigma$ linking $w_\sigma'$ to $w_\sigma$. Then note that the element
$$ \beta = c - c' - \sum_\sigma(-1)^{\dim\sigma}\gamma_\sigma $$
is a singular $1$-chain with $\partial\beta = 0$ and we say that $c$ and $c'$ are equivalent if we have $[\beta] = 0 \in H_1(W,\Z)$ (the first singular homology group). There is a natural (free transitive) action of $H_1(W)$ on the set of combinatorial Euler structures.

\begin{rmrk}
If one fixes a $0$-cell $x$ in $W$, there is a special kind of Euler chains called \emph{spider} by Turaev. These are star-shaped 1-chains, with $x$ as a center and paths from $x$ to any cell $\sigma$, oriented accordingly: from $x$ to $w_\sigma$ if $\dim \sigma$ is even, from $w_\sigma$ to $x$ if $\dim \sigma$ is odd. It is not difficult to see that any Euler chain is homologous to a spider.

A spider gives a path from $x$ to any cell $\sigma$ in $W$. In turn, it defines a lift $\widetilde \sigma$ of each cell
$\sigma$ to the universal cover $\widetilde W$ of $W$, and allows to define unambiguously the Reidemeister--Turaev torsion as in \cref{subsec:Defi}. Note that two homologous spiders will not necessarily define the same lifts in $\widetilde W$, nevertheless the value of the torsion will not be affected, since it is multiplied by the determinant of an element of the commutator subgroup $\ker(\pi_1(W) \to H_1(W))$. In fact, an Euler structure defines unambiguously a lift of each cell of $W$ to its maximal abelian cover $\overline W$.
\end{rmrk}

\begin{example}\label{ex:circle2}
For the circle, with the choice of lifts $\widetilde c^0$ and $\widetilde c^1$ we made in  \cref{circle}, the torsion equals $\pm \frac 1 {\det(\rho(u)-I_n)}$. But if one moves the lift $\widetilde c^1$ to $u^m \cdot \widetilde c^1$, then the torsion will be multiplied by $\det \rho(u^m)$. 
On the other hand, if one translates {simultaneously} $\widetilde c^0$ and $\widetilde c^1$, the torsion is not affected.
What we computed is a natural choice of torsion, for a natural induced Euler structure: one chooses \emph{any} lift $\widetilde c^0$ of the base point, and then take $\widetilde c^1$ as on \cref{circle}, taking the edge that goes out of $\widetilde c^0$ following the orientation. Any such choice induces the same Euler structure.
We denote this Euler structure on the circle by $\mathfrak e_{\circ}$.
\end{example}

\subsubsection{Smooth Euler structure}
On the other hand, a \emph{smooth} Euler structure is a homology class of nowhere vanishing vector fields. Two nowhere vanishing vector fields $\mathfrak X$ and $\mathfrak X'$ on a 3-manifold $X$ are homologous if there exists a ball $D \subset X$ such that $\mathfrak X$ and $\mathfrak X'$ are homotopic in $X \setminus D$ as nowhere vanishing vector fields. There is a natural action of $H_1(X)$ on the set of smooth Euler structures by Reeb surgery, see \cite[Section I.4.2]{Turaev_tors}.

The obstruction for two vector fields $\mathfrak X$ and $\mathfrak X'$ of being homologous is the Chern--Simons class $cs(\mathfrak X, \mathfrak X') \in H_1(X)$.
It can be defined as follows: denote by $p \colon X \times [0,1] \to X$ the projection, and choose $\mathbb X$ a smooth section of the bundle $p^*TX \to X \times [0,1]$ which provides a homotopy between $\mathfrak X = \mathbb X\vert_{X\times\{0\}}$
and $\mathfrak X' = \mathbb X\vert_{X\times\{1\}}$ and which is transverse to the zero section. The intersection of $\mathbb X$ with the zero section projects onto a one-dimensional submanifold in $X$, whose homology class is $cs(\mathfrak X, \mathfrak X')$. It can be shown that it only depends on the homology classes of $\mathfrak X$ and $\mathfrak X'$.

It turns out that both sets of combinatorial and smooth Euler structures are affine sets on $H_1(X)$, and Turaev shows in \cite[Section 6]{Turaev} that there is an $H_1(X)$-equivariant isomorphism from one to the other. 
The map from the set of combinatorial Euler structures into the set of smooth Euler structures is constructed as follows (more details in \cref{subsubsec:Euler}): first there is a ``Stiefel'' vector field $\mathfrak X$ associated to a CW complex. 
Singularities of this vector field are (a choice of) points in the interior of each cell, and the vector field pushes each higher dimensional cell onto lower dimensional ones.\\
Then given a combinatorial Euler structure, one can modify the Stiefel vector field along an associated Euler chain, so that the resulting vector field is non-singular, and yields a smooth Euler structure. In fact the Euler structure can be chosen to be a spider, and a neighborhood of this spider is homeomorphic to a ball $D \subset X$. One can see that there is an essentially unique way to extend the non-singular vector field given by $\mathfrak X$ on $X \setminus D$. Indeed, the Euler structure induced by $\mathfrak X$ depends only on $\mathfrak X$ on $X \setminus D$.\\
Turaev shows equivariance of this map under the action of $H_1(M)$, which implies its bijectivity.

\subsubsection{Reidemeister--Turaev torsion}
\label{subsec:RTT}
For a CW complex $W$ with an acyclic representation $\rho\colon \pi_1(X) \to \GL(V_\rho)$ and an Euler structure $\mathfrak e$, we define the Reidemeister--Turaev torsion 
$$\tor(W, \rho, \mathfrak e) \in \C/\{\pm 1\}$$
as the torsion of the complex $C_*(W, V_\rho)$ in the basis $\mathbf{c}^*$ given by the choice of lifts of cells induced by any spider in the class $\mathfrak e$.
For a compact manifold $M$, the torsion is a well-defined topological invariant, defined up to sign: it does not depend on the choice of a cell decomposition.

\subsubsection{Sign-refined torsion}
The fact that the torsion defined in \cref{subsec:RTT} is independent of the given choices is not too difficult to see. For example, the choice of the basis of $V_\rho$ does not affect it, since changing a basis $\mathbf v$ for a new basis $\mathbf v'$ of $V_\rho$ would change the torsion by a factor $[\mathbf v' \colon \mathbf v]^{\chi(W)}$, but the Euler characteristic of $W$ must vanish from the acyclicity assumption.

A problem that one cannot avoid facing is the sign: the basis $\mathbf c^*$ given by a choice of lift of the cells comes \emph{ordered}, and a change of the order of the cells will result into a change of sign, at least if $\dim V_\rho$ is odd. More precisely, if one permutes two $k$-cells $c^k_i$ and $c^k_j$, then the Euler structure $\mathfrak e$ will give a new basis $\mathbf d^*$, and the torsion $\tor(W, \rho, \mathbf d^*, \mathfrak e)$ will differ from $\tor(W, \rho, \mathbf c^*, \mathfrak e)$ by a sign $(-1)^{|i-j| \dim V_\rho}$. This is why so far we have only defined the torsion up to sign.

To remedy this sign ambiguity, following Turaev (\cite[Chapter III, Section 18]{Turaev_book}), we introduce a correcting term $\check \tau(W, \omega)$ where $\omega$ is a \emph{homology orientation} of $W$, namely an orientation of the vector space $H_*(W, \R) = \bigoplus_i H_i(W, \R)$. 

To define this correcting term, we chose a basis $\mathbf h$ of $H_*(W, \R)$ which is \emph{positively oriented} with respect to $\omega$. We already have bases $\mathbf c^i$ of $C_i(W, \R)$, and we chose arbitrary bases $\mathbf b^i$ of $B_i(W, \R)$. Then define
$$\check \tau(W, \omega) = (-1)^{N(C)} \, \tau(W, h) = (-1)^{N(C)} \prod_i [\mathbf b^i \sqcup \mathbf h^i \sqcup \mathbf{\overline b}^{i-1} \colon \mathbf{c}^i]^{(-1)^i}$$
where $N(C) = \sum_{i=0}^3 \left(\sum_{r=0}^i \dim C_i(W, \R)\right)\left(\sum_{r=0}^i \dim H_i(W, \R)\right)$.
Finally, define the \emph{sign refined Reidemeister-Turaev torsion} as 
$$\tor(W,\rho, \mathfrak e, \omega) = \sign(\check \tau(W, \omega))^{\dim V_\rho} \tor(W, \rho, \mathfrak e).$$
One can now see that changing the order of the cells of $W$ does not affect the torsion, since the sign change will also appear in $\sign(\check \tau(W, \omega))^{\dim V_\rho}$. Let us comment that the additional sign refinement $(-1)^{N(C)}$ is needed in order to make the formula for the torsion invariant by elementary expansions and collapses, what we will explain in the next subsection.

\begin{rmrk}
In classical textbooks such as \cite{Turaev_book, Turaev_tors}, Turaev deals with one-dimensional representations when defining the sign-refined torsion, hence there is no $\dim V_\rho$ as an exponent for the sign of the correcting term. As we already mentioned, the sign determination of the torsion only poses problem when $\dim V_\rho$ is odd, and in this case all the arguments of Turaev extend in a straightforward way. When $\dim V_\rho$ is even, there is nothing to do.
\end{rmrk}

\subsubsection{Elementary collapse, elementary expansion}

For any pair $(c^k, c^{k+1})$ of cells of dimensions $(k,k+1)$ in a CW complex $W$, such that $c^k$ appears only in the boundary of $c^{k+1}$, and the boundary map $\partial_{k+1}$ restricts to a homeomorphism $\partial_{k+1} \colon \partial_{k+1}^{-1}\{c^k\} \xrightarrow{\sim}c^k$, one can remove the pair $(c^k, c^{k+1})$ and reduce the complexity of $W$. This induces a deformation retract $W \to W' = W \setminus \{c^k, c^{k+1}\}$, which is called an \emph{elementary collapse}. An Euler structure $\mathfrak e$ on $W$ obviously induces an Euler structure $\mathfrak e'$ on $W'$. The reverse operation is called \emph{elementary expansion}.
Now we describe how  elementary collapses and elementary expansions act on the Reidemeister--Turaev torsion.

\begin{lem}
\label{lem:torcollapse}
Let $W$ be a CW complex with an acyclic representation $\rho\colon \pi_1(X) \to \GL(V_\rho)$, fix an Euler structure $\mathfrak e$ and a homology orientation $\omega$. Denoting by $\overline c^{k+1}, \overline c^k$ the lifts of the cells $c^{k+1}, c^k$ in the maximal abelian cover $\overline W$ of $W$ induced by $\mathfrak e$,  let $h$ be the unique element  in  $H_1(W)$ such that $h \overline c^k$ appears in the boundary of $\overline c^{k+1}$. Then the following holds:
$$\tor(W, \rho, \mathfrak e, \omega) = (\det \rho(h))^{(-1)^k} \tor(W', \rho, \mathfrak e', \omega).$$
\end{lem}
\begin{proof}
Let $\{v_1, \ldots, v_n\}$ be  basis of $V_\rho$.
Pick  bases $\mathbf{b}^i$ of the spaces $B_i(W,V_\rho)$ such that $\widetilde c^{k+1} \otimes v_1, \ldots \widetilde c^{k+1} \otimes v_n$ are basis vectors given by some lift $\mathbf{\overline b}^{k}$ of $\mathbf{b}^k$. Since these are also basis vectors of $C_{k+1}(W, V_\rho)$, up to some permutation the change of basis matrix $\mathbf{b}^{k+1} \sqcup \mathbf{\overline b}^{k} \colon \mathbf{c}^{k+1}$ has the form 
$$\mathbf{b}^{k+1} \sqcup \mathbf{\overline b}^{k} \colon \mathbf{c}^{k+1} = \begin{pmatrix} I_n & 0 \\ 0 & \ast \end{pmatrix}.$$
In particular its determinant remains unchanged after the elementary collapse $W \to W'$.

Now the change of basis matrix $\mathbf{b}^{k} \sqcup \mathbf{\overline b}^{k-1} \colon \mathbf{c}^{k}$ has the form, up to permutation, 
$$\mathbf{b}^{k} \sqcup \mathbf{\overline b}^{k-1} \colon \mathbf{c}^{k} = \begin{pmatrix} \rho(h \gamma) & \ast \\ 0 & \ast \end{pmatrix}$$
for some element $\gamma$ in the commutator subgroup  $[\pi_1 W, \pi_1 W]$. In particular its determinant changes by a factor $(\det \rho(h))^{(-1)^k}$. Since no other determinants are affected by the elementary collapse, the lemma is proved, up to sign.
The fact that the elementary collapse does not affect the sign of the refined torsion is proved in \cite[Lemma 18.4]{Turaev_book}.
\end{proof}

\subsubsection{Example of the solid torus}
\label{subsubsec:extorus}
Let $N = D^2 \times S^1$ be a 3-manifold homeomorphic to a solid torus. We choose a generator $h$ of $H_1(N) \simeq \Z$ (note that it induces a homology orientation on $N$). Let $\mathfrak X$ be the vector field on $N$ everywhere tangent to the circles $\{x \} \times S^1$, directed positively with respect to $h$.

In this subsection we will compute the torsion of $N$ with respect to the Euler structure $\mathfrak e_+$ induced by $\mathfrak X$. 

\paragraph{A CW structure on $N$.}
To start we choose a CW-decomposition for $N$. Take the obvious CW-decomposition of $D^2$ with one 0-cell, one 1-cell and one 2-cell. Then we take the product with the $S^1$ factor (seen as a CW-complex with one 0-cell and one 1-cell). We obtain (see also \cref{fig:solidtorus}):
\begin{itemize}
\item One 0-cell $e_0$,
\item Two 1-cells: $e_1^1$ which is the product of the 1-cell of $D^2$ with the 0-cell of $S^1$, and $e_1^2$, product of the 0-cell of $D^2$ with the 1-cell of $S^1$,
\item Two 2-cells: $e_2^1$ which is the product of the 2-cell of $D^2$ with the 0-cell of $S^1$, and $e_2^2$, product of the 1-cell of $D^2$ with the 1-cell of $S^1$,
\item One 3-cell $e_3$, product of the  2-cell of $D^2$ with the 1-cell of $S^1$.
\end{itemize}

\begin{figure}[h]
\begin{center}
\def\svgwidth{0.7\columnwidth}
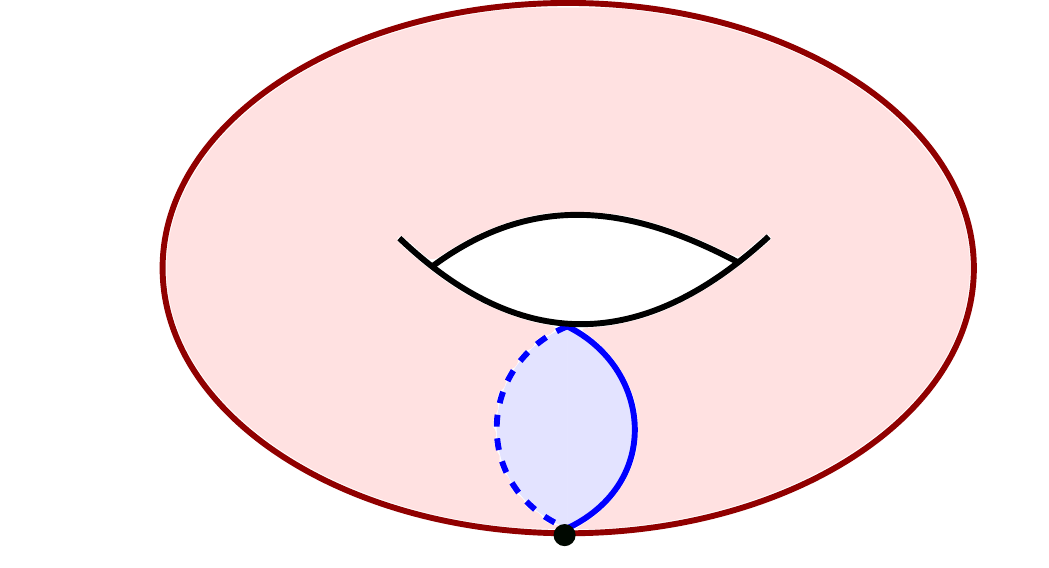
\caption{\label{fig:solidtorus} The solid torus $N$, with a cell decomposition. The 0-cell $e_0$ is below. The two 1-cells $e_1^1$ and $e_1^2$ are blue and red respectively. Then the two 2-cells are $e_2^1$ (the light blue internal disk) and $e_2^2$ (the disk given by the complement of $e_1^1 \cup e_1^2$ in the toral boundary $\partial N$ of $N$). The 3-cell $e_3$ (light red) is the complement of the disk $e_2^1$ in the interior of $N$. In green, a spider inducing  the same Euler structure as the vector field $\mathfrak X$ (in purple).}
\end{center}
\end{figure}

\paragraph{Euler structure on $N$.}
As explained by Turaev in \cite[Section VI.2.4]{Turaev_tors}, there are two natural Euler structures $\mathfrak e_{\pm}$. They are uniquely determined by the fact that $\mathfrak e_+ = h \mathfrak e_-$. We claim that the Euler structure induced by $\mathfrak X$ is $\mathfrak e_+$. It can be seen as follows: first one needs to make $\mathfrak X$ transverse to the boundary $\partial N$, pointing outward, by a homotopy. Then one can also make $-\mathfrak X$ transverse, outward pointing. Then the claim follows since it can be checked that performing a Reeb surgery as in \cite[Section I.4.2]{Turaev_tors} on $- \mathfrak X$ along $h$ yields $\mathfrak X$.

Now we describe the combinatorial Euler structure corresponding to $\mathfrak e_+$ in terms of a spider, which we have also drawn in \cref{fig:solidtorus}. This spider can be described as follows: it is centered at $e_0$, and there is a path (oriented toward $e_0$ if its other extremity has odd dimension, and out of $e_0$ if its other extremity has even dimension) joining $e_0$ to a chosen point in any other cells of the decomposition of $N$. We go
\begin{itemize}
\item from $e_1^1$ to $e_0$ following $e_1^1$ in the opposite orientation, 
\item from $e_1^2$ to $e_0$ following $e_1^2$ in the opposite orientation,
\item from $e_0$ to $e_2^1$ in $e_2^1$,
\item from $e_0$ to $e_2^2$ in $e_2^1$, going in the direction given by $h$,
\item from $e_3$ to $e_0$ in $e_3$, going in the opposite of the direction given by $h$.
\end{itemize}

\begin{lem}
\label{lem:spid}
The spider described above induces the Euler structure $\mathfrak e_+$.
\end{lem}

\begin{proof}
To see that this spider induces the Euler structure $\mathfrak e_+$, we use the fact that there is an involution on the set of combinatorial Euler structure, defined as follows:
take $\mathfrak e$ a combinatorial Euler structure, and realize it as a spider centered at $e_0$. Draw this spider in the maximal abelian cover $\overline N \simeq D^2 \times \R$, and then reflects $\overline N$ along the disk $D^2 \times \{ 0 \}$. It results a new spider, hence a new Euler structure $-\mathfrak e$. This involution corresponds to $\mathfrak X \mapsto -\mathfrak X$ through the $H_1(N)$ equivariant bijection between combinatorial and smooth Euler structures, because reversing the component of $\mathfrak X$ tangent to the $D^2 \times \{\ast \}$ disks yields a homotopic vector field, and this last procedure is precisely what is needed for $-\mathfrak X$ being also transverse, outward pointing, to the boundary.

We are reduced to prove that the spider described above induces an Euler structure $\mathfrak e$ satisfying $\mathfrak e = h \cdot (-\mathfrak e)$. Now $-\mathfrak e$ differs from $\mathfrak e$ by the following modifications:
the three paths linking $e_0$ to $e_1^2, e_2^2$ and $e_3$ go now symmetrically along the direction induced by the reversed orientation of the circle $S^1$. The difference between $\mathfrak e$ and $-\mathfrak e$ is hence $h-h+h = h$, and it proves the lemma. 
\end{proof}

\paragraph{The Reidemeister--Turaev torsion of $(N,\rho,\mathfrak e_+, \omega)$}
Now we compute the torsion of $N$, given a representation $\rho \colon \pi_1(N) \to \GL(V_\rho)$ (in other word, a matrix $H$ associated to the element $h \in H_1(N)$). Recall that $\omega$ is the homology orientation associated with the choice of a generator $h$ of $H_1(N)$.

 We won't proceed by computing directly the torsion of the CW complex described above, because this strategy would become intractable later, when we will deal with manifolds of the form $\Sigma \times S^1$ for more complicated surfaces $\Sigma$. Instead, we will proceed by \emph{elementary collapse}.

Consider the cells $\{e_3, e_2^2\}$. The new CW-complex $W' = W \setminus \{e_3, e_2^2\}$ is an elementary collapse of $W$. Now $W' \to W'' = W' \setminus \{e_2^1, e_1^1\}$ is also an elementary collapse, and we can compute the torsion on the CW-complex $\{e_0, e_1^2\}$, which is topologically a circle. 
Note that for the Euler structure $\mathfrak e_+$, one has $\partial_3 \widetilde e_3 = \widetilde e_2^2 + (1-h) \widetilde e_2^1$, and $\partial_2 \widetilde e_2^1 = \widetilde e_1^1$. Hence applying \cref{lem:torcollapse}, one can compute $\tor(N, \rho, \mathfrak e, \omega)$ as the torsion $\tor(W'', \rho, \mathfrak e_+'', \omega)$ of the core circle of $N$, as in \cref{ex:circle2}. Noticing that the induced Euler structure $\mathfrak e_+''$ is nothing but what we had denoted by $\mathfrak e_\circ$ there, we have shown:
$$\tor(N,\rho, \mathfrak e_+) =  \pm \frac 1 {\det(H -I_n)}.$$

Now let us compute the sign. Note that $H_*(W'', \R) \simeq C_*(W'', \R)$ generated by $e_0$ and $e_1^2$. Moreover one has $\partial e_1^2 = 0$. Hence 
$$\check \tau(W'', \omega)= (-1)^{(1^2 + 2^2)} \frac{[\mathbf b^0 \sqcup \mathbf h^0 \colon \mathbf c^0]}{[\mathbf b^1 \sqcup \mathbf h^1 \sqcup \mathbf{\overline b}^0 \colon \mathbf c^1]}\\
=- \frac {[e_0\colon e_0]}{[e_1 \colon e_1]} = -1.
$$
Inserting the correction $(-1)^{\dim V_\rho}$ in the torsion, we get
\begin{prop}
We have
\label{prop:torstorus}
$$\tor(N,\rho, \mathfrak e_+, \omega) =  \frac 1 {\det(I_n - H)}.$$
\end{prop}

\subsection{The torsion of the unit tangent bundle}
\label{subsec:Torsion}
Let $X$
be a surface of genus $g$ with $r$ orbifold singularities, such that $2g-2+r>0$. Let $X_1$
be the unit tangent bundle of the orbisurface $X$. We denote by $c_1, \ldots, c_r\in\pi_1(X_1)$ the elliptic elements given by loops in $X$ around the singularities.

The aim of this section is to prove the following theorem.
\begin{thm}
\label{prop:torsion}
Let $\rho \colon \pi_1(X_1) \to \GL(V_\rho)$ be an irreducible representation of dimension $n$, with $\rho(u) \neq I_n$. Then $\rho$ is acyclic. For any hyperbolic metric on $X$, the geodesic flow on $X_1$ induces an Euler structure $\mathfrak e_{\geod}$ for which
$$\tor(X_1, \rho, \mathfrak e_{\geod}, \omega^1) =  \frac{\det(I_n - \rho(u))^{2g+r-2}}{\prod_{j=1}^r \det(I_n - \rho(c_j))}$$
\end{thm}

The proof of this theorem will occupy the rest of this section. It is subdivided into three parts: first in \cref{subsubsec:cell}, we decompose the unit tangent bundle $X_1$ into simpler pieces and exhibit a cell decomposition of each piece together with a natural choice of lifts of the cells in the universal cover. We prove \cref{prop:torsion} for this choice, up to sign.
In \cref{subsubsec:Euler}, we show that this natural choice of lifts corresponds to the Euler structure induced by the geodesic flow induced by any hyperbolic metric.  In \cref{subsubsec:sign} we compute the sign.
\medbreak

\subsubsection{Computation of the torsion}
\label{subsubsec:cell}

\paragraph{Decomposition of $X_1$}
The 3-manifold $X_1$ can be cut along tori into pieces
\begin{equation}
\label{eq:paste}
X_1 = X_0 \cup N_0 \cup N_1 \cup \ldots \cup N_r,
\end{equation}
where
\begin{itemize}
\item 
$X_0 \simeq (\overline X \setminus D_0) \times S^1$ is the product of a surface with $r+1$ boundary components with a circle;
\item $N_0$ is a solid torus $D_0 \times S^1$;
\item $N_i$ are solid tori containing the exceptional fibers which correspond to the elliptic elements $c_j$.
\end{itemize}

We will use the following cell decomposition of $X_0$:
the surface $\overline X \setminus D_0$ retracts on a bouquet of $2g+r$ circles, each of which can be identified with one of $a_1, \ldots, b_g, c_1, \ldots, c_r$. We denote the circles by $A_1, B_1, \ldots, A_g, B_g, C_1, \ldots, C_r$, each corresponding to the element of the fundamental group of $\overline X \setminus D_0$ denoted by the same tiny letter.
Now the product cell structure on $X_0$ is given by
\begin{itemize}
\item One 0-cell $e_0$, the product of the 0-cell of the surface with the 0-cell of the circle $S^1$;
\item $2g+r+1$ one-cells: $A_1, B_1, \ldots, A_g, B_g, C_1, \ldots, C_r$ which are products of the 1-cells of the surface with the 0-cell of $S^1$ and $U$, the product of the 0-cell of the surface with the 1-cell of $S^1$;
\item $2g + r$ two-cells $E_{A_1}, \ldots, E_{B_g}, E_{C_1}, \ldots, E_{C_r}$, products of 1-cells in the surface with the 1-cell of $S^1$. Note that $\partial E_{\gamma} = \gamma U \gamma^{-1} U^{-1}$ for any $\gamma \in \{A_1, \ldots, C_r\}$.
\end{itemize}
The retraction of the surface $\overline X \setminus D_0$ onto the bouquet of circle induces a retraction of the product $X_0$ onto this CW complex, which is topologically a union of $2g+r$ tori, all of them glued along the circle denoted by $U$.

As we will see, this retraction can be realized by elementary collapses, in particular we can compute the torsion on the $2$-dimensional complex described above.

\paragraph{Natural lift of the cell structure.}
We pick an arbitrary basis $\{v_1, \ldots, v_n\}$ of $V_\rho$, then the complex $C_*(X_0, V_\rho)$ can be computed explicitly, as in \cite[Section 4]{Kitano}, but we need to fix lifts of the cells of $X_0$ in the universal cover $\widetilde X_0$. Such a choice of lifts will induce a homology class of spider, hence an Euler structure, that we will denote by $\mathfrak e_0$.

 Since $X_0$ has the structure of a bouquet of circles times a circle, its universal cover has the structure of a union of lattices glued along a line, a portion of one of them is drawn in \cref{cover}.

\begin{figure}[h]
\begin{center}
\def\svgwidth{0.7\columnwidth}
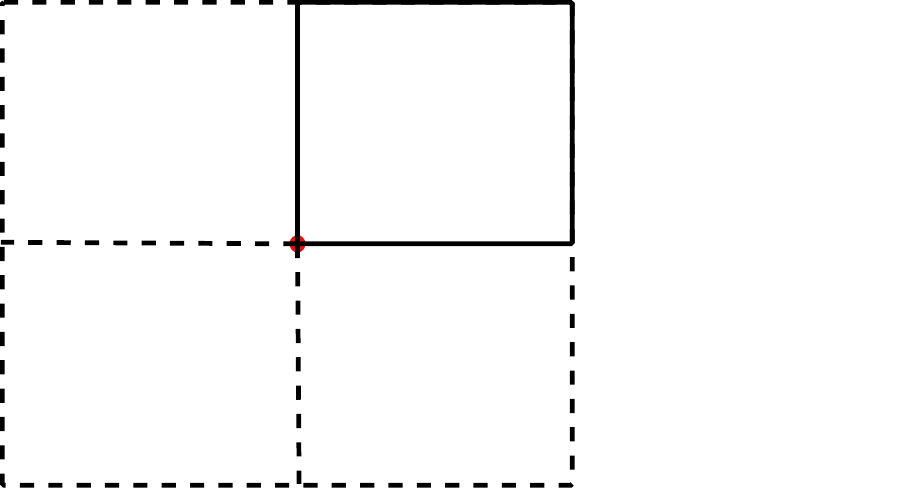
\caption{\label{cover} A sample of the universal cover of $X_0$. Here we draw part of a lattice corresponding to the universal cover of the two-cell $E_{A_1}$. The Euler structure $\mathfrak e_0$ is induced by the lifts $\widetilde e_0, \widetilde U, \widetilde A_1$ and $\widetilde E_{A_1}$ for each such two-dimensional cell.}
\end{center}
\end{figure}

We start with an arbitrary lift $\widetilde e_0$ of the unique one-cell $e \in C_0(X_0)$. Then we choose the lifts $\widetilde A_1, \widetilde U$ and $\widetilde E_{A_1}$ as in \cref{cover}. Similarly, we take the lifts of $B_1, A_2, \ldots, C_r$ and of the corresponding two-cells as above.

We denote by $\mathfrak e_0$ the Euler structure induced by this choice of lifts on the CW complex $W$ representing $X_0$. On the peripheral tori $N_i$, $i=0, \ldots, r$, we take the Euler structure $\mathfrak e_+$ as in \cref{subsubsec:extorus}.

There is a natural gluing procedure for Euler structures (\cite[Section VI.3.1]{Turaev_tors}). On the combinatorial side, if one glues $(Y_1, \mathfrak e_1)$  to $(Y_2, \mathfrak e_2)$ along $Y$, such that the Euler structure induced by  $\mathfrak e_1$ and $\mathfrak e_2$ coincide on $Y$ as $\mathfrak e$, then the glued Euler structure is $\mathfrak e_1+ \mathfrak e_2 - \mathfrak e$ as a formal sum of homology classes of singular chains.

We denote by $\mathfrak e$ the  Euler structure on $X_1$ obtained by gluing $\mathfrak e_0$ with the union of the $\mathfrak e_+$ on each $N_i$. We prove the following:
\begin{prop}
\label{prop:torscomb}
We have 
$$\tor(X_1, \rho, \mathfrak e) = \pm \frac{\det(I_n-\rho(u))^{2g+r-2}}{\prod_{j=1}^r \det(I_n-\rho(c_j))}.$$
\end{prop}

Note that \cref{prop:torsion} will follow from this proposition after we will have shown that the Euler structure $\mathfrak e$ coincides with the Euler structure $\mathfrak e_\geod$ on $X_1$, and after determining the sign. This will be done in the next subsections.

Before we prove \cref{prop:torscomb}, let us describe the matrices of the boundary maps of the complex $C_*(X_0, V_\rho)$ given by our choice of bases.

The set $\{ v_i \otimes \widetilde E_\gamma \mid i = 1, \ldots, n, \, \gamma = A_1, \ldots, C_r \}$ is now a basis of $C_2(X_1, V_\rho)$. Similarly, $\{v_i \otimes \gamma \mid i = 1, \ldots, n, \, \gamma = \widetilde A_1, \ldots, \widetilde C_r, \widetilde U \}$ is a basis of $C_1(X_1, V_\rho)$ and $\{v_i \otimes \widetilde e_0 \mid i= 1, \ldots, n\}$ is a basis of $C_0(X_1, V_\rho)$.

The boundary operator $\partial_2 \colon C_2(X_1, V_\rho) \to C_1(X_1, V_\rho)$ reads
$$\partial_2(v_i \otimes E_{A_1}) = v_i \otimes \left((1-u) \cdot \widetilde A_1 + (a_1 - 1) \cdot \widetilde U\right),$$ and similarly for other two-cells.
Further, $\partial_1(v_i \otimes \widetilde A_1) = v_i \otimes (a_1 - 1) \widetilde e_0$ and similarly for the other one-cells. Hence, in the given bases, the boundary maps are given by the following matrices
$$\partial_2 = \begin{pmatrix} 
I_n - \rho(u) & 0 & \cdots & \cdots & 0\\
0 & I_n - \rho(u) &0 &  \cdots & 0\\
\vdots & \vdots & \vdots & \vdots  & \vdots\\
0 & 0 & \cdots & \cdots &  I_n - \rho(u)\\
\rho(a_1) - I_n & \cdots & \cdots & \cdots  & \rho(c_r)-I_n
\end{pmatrix},$$
$$\partial_1 = \begin{pmatrix} \rho(a_1) - I_n & \cdots & \cdots & \cdots  & \rho(c_r)-I_n & \rho(u)-I_n \end{pmatrix}.$$

We start with the following two preparatory lemmas:
\begin{lem}
\label{lem:bound}
Let $T = S^1\times S^1$ be a torus, $\rho \colon \pi_1(T) \to \GL(V_\rho)$ a representation such that $\det(\rho (\gamma) -I_n) \neq 0$ for some $\gamma$. Let $\mathfrak e$ be the Euler structure on $T$ induced by the choice of lifts described in \cref{cover}. Then $C_*(T, V_\rho)$ is acyclic and $\tor(T, \rho, \mathfrak e) = \pm 1$.
\end{lem}
\begin{proof}
Take $\gamma$ being the homotopy class of one of the two $S^1$ factors, and denote by $\mu$ the class of the other factor. Then the complex $C_*(T,V_\rho)$ is
$$C_2(T, V_\rho) \simeq V_\rho \xrightarrow{\partial_2} C_1(T, V_\rho) \simeq V_\rho^2 \xrightarrow{\partial_1} C_0(T, V_\rho) \simeq V_\rho $$
where 
$$\partial_2 = \begin{pmatrix} \rho(\gamma) - I_n\\ I_n - \rho(\mu) \end{pmatrix}, \quad \partial_1 = \begin{pmatrix} \rho(\gamma) - I_n &  \rho(\mu)- I_n \end{pmatrix}$$
Since $\det (\rho(\gamma)-I_n) \neq 0$, one has $\partial_2$ is injective, $\partial_1$ is surjective and since $\chi(T)=0$, one deduces that the complex $C_*(Y, V_\rho)$ is acyclic. If one choses bases $\mathbf b^i$ of the spaces $B_i(T, V_\rho)$ such that the change of basis matrices $\mathbf{\overline b}^1 \colon \mathbf c^2 = \mathbf{ b}^0 \colon \mathbf c^0 = I_n$, then one has

$$\mathbf b^1 \sqcup \mathbf{\overline b}^0  \colon \mathbf c^1 = \begin{pmatrix} \rho(\gamma)-I_n & 0 \\ 0 & (\rho(\gamma)-I_n)^{-1} \end{pmatrix}$$
and the lemma follows.
\end{proof}
\begin{lem}
Let $\rho \colon \pi_1(X_1) \to \GL(V_{\rho})$ be an $n$-dimensional irreducible representation. If $\det (\rho(u)-I_n) \neq 0$, then
the complex $C_*(X_1, V_{\rho})$ is acyclic.
\end{lem}

\begin{proof}
For $i=0, \ldots, r$, the solid torus $N_i$ has the homotopy type of a circle, hence the complex $C_*(N_i, V_{\rho})$ is just
$$ C_1(N_i, V_{\rho}) \simeq V_{\rho} \xrightarrow{\rho(c_i)-I_n} V_{\rho} \simeq C_0(N_i, V_{\rho})$$
where $\rho(c_0)$ is understood as being $\rho(u)$. Since $\rho(c_i)^{\nu_i} = \rho(u)$, the assumption $\det (\rho(u)-I_n) \neq 0$ implies $\det (\rho(c_i)-I_n) \neq 0$ for any $i$.
Hence, we showed that $C_*(N_i, V_{\rho})$ is acyclic when $\det (\rho(u)-I_n) \neq 0$. Also $C_*(\partial N_i, V_\rho)$ is acyclic for any $i$, see \cref{lem:bound}.

Now we prove that $C_*(X_0, V_{\rho})$ is acyclic. First, the map $\partial_1$ is surjective since $\det (\rho(u)-I_n) \neq 0$, hence $H_0(X_0, V_\rho) = \{0\}$. Moreover, the $n(2g+r) \times n(2g+r)$ diagonal submatrix 
$$\partial'_2 = \begin{pmatrix} 
I_n - \rho(u) & 0 & \cdots & \cdots & 0\\
0 & I_n - \rho(u) &0 &  \cdots & 0\\
\vdots & \vdots & \vdots & \vdots  & \vdots\\
0 & 0 & \cdots & \cdots &  I_n - \rho(u)\\
\end{pmatrix}$$
of $\partial_2$ has determinant equal to $\det(I_n - \rho(u))^{2g+r}$, in particular it is non-zero and $\partial_2$ is injective, i.e., $H_2(X_0, V_\rho)=\{0\}$. Since the Euler characteristic of $X_0$ vanishes, we conclude that $H_1(X_0, V_\rho) = \{0\}$ and the acyclicity follows. By a Mayer--Vietoris argument one deduces that $C_*(X_1, V_\rho)$ is acyclic.
\end{proof}

\begin{proof}[Proof of  \cref{prop:torscomb}]
First, because $\rho(u) \neq I_n$ is central, $\det(I_n - \rho(u)) \neq 0$.
Now we use again a Mayer--Vietoris argument. 
For any solid torus $N_i$ in the decomposition given in \cref{eq:paste}, we compute the torsion as in \cref{prop:torstorus} in the Euler structure $\mathfrak e_\circ = \mathfrak e_+$. Denote by $\ell_i$ the core of the solid torus $N_i$. Since $c_i^{\nu_i} u^{-1}$ is a meridian of $N_i$, the curve $\ell_i$ is given by $\ell_i= c_i^{n_i} u^{m_i}$ with $m_i\nu_i + n_i = 1$. Writing $u=c_i^{\nu_i}$, one sees that $\ell_i = c_i$ is the core of the solid torus $N_i$. For the torsion we obtain
$$\tor(N_i, \rho, \mathfrak e_\circ) = \pm \frac 1 {\det(I_n - \rho(c_i))},$$
 where $c_0$ is understood as being $u$. 
On the other hand, using \cite[Theorem 2.2]{Turaev_book}, one computes the torsion of $X_0$ in the Euler structure $\mathfrak e_0$:
\begin{equation}
\label{eq:torX0}
\tor(X_0, \rho, \mathfrak e_0) = \pm \frac{\det(\partial_2')}{\det(\rho(u)-I_n)}
\end{equation}
and it follows 
\begin{equation}
\label{eq:tor0}
\tor(X_0, \rho, \mathfrak e_0) = \pm \det(I_n-\rho(u))^{2g+r-1}.
\end{equation}

Finally, we use the multiplicativity of the torsion to compute the torsion of $X_1$ from the torsions of the pieces $X_0, N_i$ and the Mayer--Vietoris sequence.
Note that the Euler structure induced by $\mathfrak e_0$ on $\partial X_0$ coincides with the Euler structure induced by the Euler structure $\mathfrak e_+$ (see \cref{subsubsec:extorus}) on each toral boundary $\partial N_i$.
We use the exact sequence of complexes
$$0 \to C_*(\partial X_0, \rho, \mathfrak e_0) \to C_*(\coprod_i N_i, \rho, \mathfrak e_+) \oplus C_*(X_0, \rho, \mathfrak e_0) \to C_*(X_1, \rho, \mathfrak e) \to 0$$
 We have then from \cite[Theorem 1.5]{Turaev_book} the formula:
$$\tor(X_1, \rho, \mathfrak e) =\pm \frac{\tor(X_0, \rho, \mathfrak e_0) \prod_{i=0}^r\tor(N_i, \rho, \mathfrak e_\circ)}{\prod_{i=0}^r\tor(\partial N_i,\rho, \mathfrak e_0)}.$$

For the boundary tori $\partial N_i$, \cref{lem:bound} gives $\tor(\partial N_i, \rho, \mathfrak e_0) = 1$. We deduce
\begin{equation*}
	\tor(X_1, \rho, \mathfrak e) = \pm \frac{\det(I_n-\rho(u))^{2g+r-2}}{\prod_{j=1}^r \det(I_n-\rho(c_j))}.\qedhere
\end{equation*}
\end{proof}

\subsubsection{Computation of the Euler structure}
\label{subsubsec:Euler}
Now we prove that the Euler structure $\mathfrak e$ on $X_1$ that we have chosen in \cref{subsubsec:cell} corresponds to the geodesic flow through Turaev's isomorphism between combinatorial and smooth Euler structure.

First, note that the geodesic flow $\mathfrak X_\geod$ on $X_1$ is homotopic to the vector field $\mathfrak X_{\rot}$
obtained by ``rotation in the fiber'' defined at the end of \cref{subsec:Orbi}. A nowhere vanishing homotopy can be realized as $t  \mathfrak X_\rot + (1-t) \mathfrak X_{\geod}$ on $X_0$, since they generate a 2-dimensional plan everywhere.

In particular it follows from \cref{subsubsec:extorus} that the Euler structure induced by $X_1$ coincides with the Euler structure $\mathfrak e_+$ on each peripheral solid tori $N_i$.

For $X_0$, by elementary expansion we can turn the CW decomposition of \cref{subsubsec:cell} into a product cell structure, with one 3-cell:
first add a 1-cell $C_0$ (corresponding to $c_0$) and a 2-cell $E_0$ to the bouquet of circle which was the cell decomposition for the surface $\overline X \setminus D_0$, such that the boundary of the two cell $E_0$ yields the relation $\prod[a_i, b_i]\prod c_j = c_0$. Now taking the product with the circle $S^1$ turns into adding one more 2-cell $E_{C_0}$ and a 3-cell $B_0$ to the CW complex $W$ that we described for $X_0$. We call the new CW complex $W'$. We extend the Euler structure $\mathfrak e_0$ on $W$ into an Euler structure $\mathfrak e_0'$ on $W'$ in the same way we defined the Euler structure for the solid torus: 
\begin{itemize}
\item
we take a path from an inner point in $C_0$ to $e_0$ following $C_0$ opposite to its orientation, 
\item
a path from $e_0$ to an inner point in $E_0$, staying in the interior of $E_0$,
\item
a path from $e_0$ to an inner point in $E_{C_0}$, staying in the interior of $E_{C_0}$, and going in the direction given by the circle factor,
\item
a path from an inner point in  $B_0$ to $e_0$, staying in the interior of $B_0$, going in the direction opposite to the circle.
\end{itemize}
Applying \cref{lem:torcollapse}, we have
$$\tor(W', \rho, \mathfrak e_0') = \tor(W, \rho, \mathfrak e_0)$$
hence we are led to show that the Euler structure $\mathfrak e'_0$ is the same as $\mathfrak e_\rot$, the one induced by the rotation flow $\mathfrak X_\rot$ on $W'$. We recap this here as a proposition:

\begin{prop}\label{prop:euler}
The Euler structures $\mathfrak e'_0$ and $\mathfrak e_\rot$ coincide.
\end{prop}

\begin{proof}
As a first step, let us prove the following claim:
\begin{claim}
The difference $\mathfrak e'_0 - \mathfrak e_\rot$ is represented by a multiple of the homology class of $u$ in $H_1(X_0)$.
\end{claim}
To see this, consider the action of the group of diffeomorphisms of the surface $\overline X \setminus D_0$ on the manifold $X_0 = (\overline X \setminus D_0) \times S^1$, acting on each slice.
This induces an action on the set of Euler structure, and on the other hand it acts on the fundamental group in the obvious way, and on the homology $H_1(\overline X \setminus D_0) \subset H_1(X_0)$. Now note that this action fixes $\mathfrak e_0'$ (since $\tor(X_0, \rho, \mathfrak e_0')$ computed in \cref{eq:tor0} is invariant for any $\rho$), and obvious fixes $\mathfrak e_\rot$ (since it fixes $\mathfrak X_\rot$). Writing $h = \mathfrak e'_0 - \mathfrak e_\rot \in H_1(X_0)$, it follows that $h$ is fixed by the action of any diffeomorphism of the surface, and the claim follows.

\medbreak

Let us denote by $\mathfrak X$ the Stiefel vector field on the CW complex $W'$. After fixing one point in the interior of each cell of $W'$, $\mathfrak X$ is a vector field that ``pushes'' each cell on the lower dimensional skeleton, and which singularities are precisely these points, one in each cell. As we mentioned, given an Euler chain $c$ representing the Euler structure $\mathfrak e'_0$, after assuming that we picked points in each cells which are precisely the endpoints of the paths in $c$, there is an essentially unique way to desingularize $\mathfrak X$ into a nowhere vanishing vector field $\mathfrak X_0$, which represents the Euler structure $\mathfrak e'_0$. 

Let us decompose the 2-cell $E_0$ into three parts: taking a radial coordinate $R$ on $E_0$ (which is homeomorphic to a 2-dimensional disk), we cut it into $\{ R\le 1/3\}$, $\{1/3 < R < 2/3\}$ and $\{2/3\le R \le 1\}$. It also gives a decomposition of $B_0 = E_0 \times S^1 \setminus E_0$. 
As Turaev explains in \cite[Section 5.2]{Turaev}, since $\mathfrak e'_0 - \mathfrak e_\rot$ is a multiple of $u$, the obstruction of $\mathfrak X_0$ and $\mathfrak X_\geod$ being homologous can be computed in a tubular neighborhood of a representative of $u$ in $X_0$, which we realize as 
$\{ R\le 1/3\} \times S^1$, the inner part on \cref{fig:X0}.

\begin{figure}[h]
\begin{center}
\def\svgwidth{1\columnwidth}
%% Creator: Inkscape 1.0.2 (e86c8708, 2021-01-15), www.inkscape.org
%% Accompanies image file 'euler.pdf' (pdf, eps, ps)
\begingroup%
  \makeatletter%
  \providecommand\color[2][]{%
    \errmessage{(Inkscape) Color is used for the text in Inkscape, but the package 'color.sty' is not loaded}%
    \renewcommand\color[2][]{}%
  }%
  \providecommand\transparent[1]{%
    \errmessage{(Inkscape) Transparency is used (non-zero) for the text in Inkscape, but the package 'transparent.sty' is not loaded}%
    \renewcommand\transparent[1]{}%
  }%
  \providecommand\rotatebox[2]{#2}%
  \newcommand*\fsize{\dimexpr\f@size pt\relax}%
  \newcommand*\lineheight[1]{\fontsize{\fsize}{#1\fsize}\selectfont}%
  \ifx\svgwidth\undefined%
    \setlength{\unitlength}{514.27249293bp}%
    \ifx\svgscale\undefined%
      \relax%
    \else%
      \setlength{\unitlength}{\unitlength * \real{\svgscale}}%
    \fi%
  \else%
    \setlength{\unitlength}{\svgwidth}%
  \fi%
  \global\let\svgwidth\undefined%
  \global\let\svgscale\undefined%
  \makeatother%
  \begin{picture}(1,0.22443249)%
    \lineheight{1}%
    \setlength\tabcolsep{0pt}%
    \put(0,0){\includegraphics[width=\unitlength,page=1]{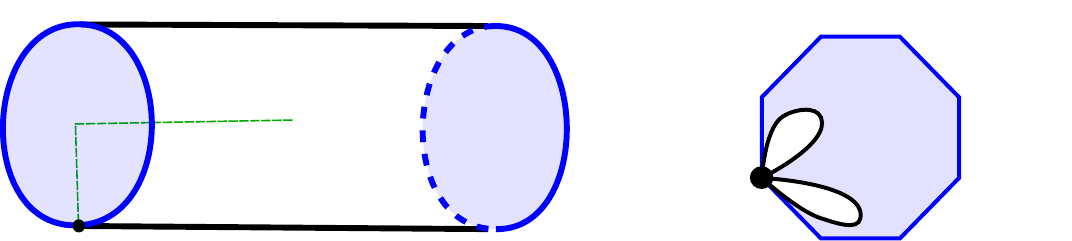}}%
    \put(0.66171373,0.08759346){\makebox(0,0)[lt]{\lineheight{1.25}\smash{\begin{tabular}[t]{l}$a_1$\end{tabular}}}}%
    \put(0.69887226,0.1691584){\makebox(0,0)[lt]{\lineheight{1.25}\smash{\begin{tabular}[t]{l}$b_1$\end{tabular}}}}%
    \put(0.78466298,0.20897732){\makebox(0,0)[lt]{\lineheight{1.25}\smash{\begin{tabular}[t]{l}$a_1^{-1}$\end{tabular}}}}%
    \put(0.87670378,0.17131993){\makebox(0,0)[lt]{\lineheight{1.25}\smash{\begin{tabular}[t]{l}$b_1^{-1}$\end{tabular}}}}%
    \put(0.79888981,0.04376411){\makebox(0,0)[lt]{\lineheight{1.25}\smash{\begin{tabular}[t]{l}$c_0$\end{tabular}}}}%
    \put(0.78466298,0.11483388){\makebox(0,0)[lt]{\lineheight{1.25}\smash{\begin{tabular}[t]{l}$E_0$\end{tabular}}}}%
    \put(0,0){\includegraphics[width=\unitlength,page=2]{euler.pdf}}%
  \end{picture}%
\endgroup%

\caption{\label{fig:X0} On the right, we describe the 2-cell $E_0$. Note that all the boundary part, and the gluing between the edges representing the $a_i$ and $b_i$, lie in the outer annulus $\{2/3\le R \le 1\}$. 
On the left we draw the 3-cell $B_0$, the light blue two cell schematically represents $E_0$ (where we omit the identifications in the boundary, and the 1-cells attached). The Euler chain representing $\mathfrak e_0'$ also lies in this outer annulus, except for the part that we drawn in green. In purple, we draw a vector field representing the same homology class as $\mathfrak X_0$.}
\end{center}
\end{figure}

In other words, the homology class of the vector field $\mathfrak X_0$ can be represented by a vector field equal to $\mathfrak X_\rot$ on the outer part $\{2/3\le R \le 1\} \times S^1$. On the inner part $\{ R\le 1/3\} \times S^1$, we remove a tubular neighborhood of an Euler chain representing $\mathfrak e_0'$: a path from $e_0$ to the center of $E_0$ in $E_0$, and then a path from the center of $B_0$ to the center of $E_0$, that we take in the interior of $\{ R\le 1/3\} \times S^1$ (in green in \cref{fig:X0}). Outside of this ball neighborhood, so from the center of $B_0$ to the other copies of $E_0$ (note $B_0$ touches $E_0$ from both sides), $\mathfrak X_0$ also equals $\mathfrak X_\rot$, by definition of the Stiefel vector field. In the intermediate part $\{1/3 < R < 2/3\} \times S^1$, we just interpolated the two by a homotopy near the two boundary parts. 
Now we see that the two vector fields are homotopic, outside of a ball neighborhood of the Euler chain, and it proves that the induced Euler structures coincide.
\end{proof}

\begin{rmrk}
Unlike the manifold $X_1$, the pieces $X_0$ and $N_i$ are manifolds with boundary. In this case, Turaev introduces a thickening of the boundary so that the vector fields representing Euler structures are transverse to the boundary. Here, we deal with vector fields tangent everywhere to the boundary. It makes the description much simpler, and we can recover the setting of Turaev by taking a thickening of the boundaries of $X_0$ and of each $N_0$ and by extending the vector fields so that \begin{itemize}
\item
Each vector field is tangent to the slice $\partial \Theta \times \{0\}$ for $\Theta \in \{X_0, N_0, \ldots, N_r\}$.
\item
The vector field is pointing outward on $\partial X_0 \times \{1\}$.
\item
The vector field are pointing inward on $\partial N_i \times \{1\}, i=0, \ldots, r$.
\end{itemize}
\end{rmrk}

\subsubsection{Sign-refined torsion}
\label{subsubsec:sign}

\paragraph{Homology orientation}
First we need to fix a homology orientation on $X_1$. Like for any closed 3-manifold, there is a canonical choice of homology orientation $\omega^1$: take any base of $H_0(X_1, \R) \oplus H_1(X_1, \R)$, and the dual base of $H_2(X_1, \R) \oplus H_3(X_1, \R)$.
We fix the following homology bases:
$H_0(X_1, \R)$ has a natural base, and take the base given by $\{A_1, B_1, \ldots, A_g, B_g, U\}$ of $H_1(X_1, \R) \simeq \R^{2g+1}$ of $H_1(X_1, \R)$.

 Let us make clear that the dual basis of $H_2(X_1, \R)$ is represented by the cells
$\{ E_{B_1}, E_{A_1}, \ldots, E_{B_g}, E_{A_g}, E_0\}$ \emph{in this order}.

To compute the sign of the torsion, we use the gluing formula (\cite[Theorem VII.1.4]{Turaev_tors}): 
\begin{equation}
\label{eq:TurGlue}
\tor(X_1, \rho, \mathfrak e, \omega^1) = \frac { \tor (X_0, \rho, \mathfrak e_0, \omega^0)}{\prod_{j=0}^r \det(I_n-\rho(c_i))}
\end{equation}
where $\frac 1 {\det(I_n-\rho(c_i))}$ is the torsion of the solid torus $N_i$ with its natural homology orientation, see \cref{prop:torstorus}.

The homology orientation $\omega^0$ is defined in \cite[Section V.2]{Turaev_tors}: there is a homology orientation 
$$\widetilde \omega = (-1)^{(r+1)b_3(X_1) + (b_1(X_0)+1)(b_1(X_1)+r+1)} \omega^1 = - \omega^1$$
of $X_1$ 
such that  for the homology orientations $\widetilde \omega$ of $X_1$, $\omega^0$ of $X_0$ and the natural homology orientation $\omega_{(X_1, X_0)}$ of the pair $(X_1, X_0)$, the torsion of the long exact sequence of the pair $(X_1,X_0)$ in real homology has positive sign.

\begin{lem}
The homology orientation $\omega^0$ of $X_0$ is induced by the bases:
\begin{itemize}
\item 
$-e_0$ of $H_0(X_0, \R)$,
\item
$\{A_1, B_1, \ldots, A_g, B_g, C_1, \ldots, C_r, U \}$ of $H_1(X_0, \R) \simeq \R^{2g+r+1}$,
\item
$\{ E_{B_1}, E_{A_1}, \ldots, E_{B_g}, E_{A_g}, E_{C_1}, \ldots, E_{C_r}\}$ of $H_2(X_0, \R) \simeq \R^{2g+r}$,
\end{itemize}
\end{lem}

\begin{proof}
Let us compute the long exact sequence of the pair $(X_1,X_0)$ in homology with real coefficients:
\begin{multline}
\label{eq:LES}
0 \to H_3(X_0) \xrightarrow{i_3}  H_3(X_1) \xrightarrow{r_3} H_3(X_1, X_0) \xrightarrow{\delta_3} \\
H_2(X_0) \xrightarrow{i_2} H_2(X_1) \xrightarrow{r_2} H_2(X_1, X_0) \xrightarrow{\delta_2}\\
H_1(X_0) \xrightarrow{i_1}  H_1(X_1) \xrightarrow{r_1} H_1(X_1, X_0) \xrightarrow{\delta_1}\\
H_0(X_0) \xrightarrow{i_0}  H_0(X_1) \xrightarrow{r_0} H_0(X_1, X_0) \to 0
\end{multline}
We need to show that the torsion of this long exact sequence with respect to the orientations $\widetilde \omega, \omega^0, \omega_{(X_1,X_0)}$ has positive sign.  We compute the sign of the torsion with the orientations $\omega^1 = -\widetilde \omega, -\omega_0, \omega_{(X_1,X_0)}$, and show it is positive.

Clearly $H_3(X_0) = \{0 \}$ since it is a 3-manifold with boundary, and also $H_1(X_1, X_0) = H_0(X_1, X_0) = \{0\}$. 
We already specified bases of $H_*(X_1)$ and of $H_*(X_0)$. As we alluded before, there is a natural basis of $H_*(X_1, X_0)$ (\cite[V.2.2]{Turaev_tors}) supported in the peripheral tori: each torus $N_i$ contributes a class $[N_i, \partial N_i]$ in $H_3(X_1, X_0)$, and a class $[D_i, \partial D_i]$ in $H_2(X_1, X_0)$, where $N_i = S^1 \times D_i$.

Let us write down the matrices of the linear maps involved in \cref{eq:LES} in these bases. The notations $\bigzero_{n\times m}$ and $\bigId_{n}$ stand for the $(n\times m)$ matrix with all entries zero, and for the $(n \times n)$ identity matrix:
$$
r_3 = \bma 1 \\ \vdots \\ 1 \ema, \quad \delta_3 = \bma & \begin{matrix} \bigzero_{2g \times (1+r)} \end{matrix} & \\
& & 0 \\
&  \begin{matrix} \bigId_r \end{matrix}& \vdots \\
& & 0
\ema, 
\quad 
i_2 =  \bma & \bigId_{2g} & & \bigzero_{2g \times r} & \\
0 &\ldots & 0 &\ldots & 0
\ema
$$
$$
r_2 = \bma & 1 \\
& 0\\
\bigzero_{(1+r) \times 2g} & \vdots\\
 & \\
 & 0
 \ema, 
 \quad
 \delta_2 = \bma 0 & & & \\
 \vdots & & \bigzero_{2g \times r} &\\
0&&& \\
 \vdots & & \bigId_r&\\
 0 & & \ldots &0 
 \ema, 
 $$
 $$
 i_1 = \bma & & & &\\
 & \bigId_{2g} & & \bigzero_{2g \times (r+1)} &\\
 & & & &\\
 0 &\ldots & \ldots & 0 & 1
 \ema, 
 \quad 
 i_3 = r_1 = \delta_1 = r_0 = 0, 
 \quad 
 i_0 = \bma 1 \ema
$$
To compute the sign of the torsion if this long exact sequence, one needs to compare the bases of the homology vector spaces involved, with the new bases given by the matrices above. As it turns out, the sign is given by the deteminant of some permutation matrices. A careful examination of the matrices involved shows that the sign is one. We leave the verification to the reader, we just mention that the only cases where the determinant is not trivially equal to 1 occur when there is a permutation matrix with a bloc $ \textbf{I}_{2g}$. Since it has even size, the all the determinants involved are indeed equal to 1. 
\end{proof}

\paragraph{Determination of the sign.}
Now we compute the sign of $\tor(X_1, \rho, \mathfrak e, \omega^1)$. We use \cref{eq:TurGlue}, hence we need to carefully follow the sign in the computation of $\tor(X_1, \rho, \mathfrak e_0)$ performed in \cref{subsubsec:cell}, with respect to the homology orientation $\omega^0$ of $X_0$. The equation to examine is \cref{eq:torX0}.
First note that the denominator $\pm (\rho(u) - I_n)$, which is given by the boundary map $\partial_1$, comes with a sign: it is now $I_n - \rho(u)$ since the basis of $C_0(X_0, \rho)$ fitting with the homology orientation $\omega^0$ is 
$\{ -v_1 \otimes \widetilde e_0, \ldots, - v_n \otimes \widetilde e_0 \}$. Now in the numerator, one needs to be careful that the matrix $\partial_2'$, which was a diagonal matrix, is now the product of $\pm(I_n - \rho(u))$ with a permutation matrix, since the basis of $C_2(X_0, \rho)$ fitting with the homology representation exchanges the $v_i \otimes \widetilde E_{A_j}$ with $v_i \otimes \widetilde E_{B_j}$. 
To determine the sign, one needs to determine the orientation of the 2-cells $\widetilde E_{A_j}$ and $\widetilde E_{B_j}$. Recall that they are dual to $\widetilde B_j, \widetilde A_j$, and the duality is given by the triple of tangent vectors to $A_j, B_j, U$ determining the positive orientation in $X_0$. Hence one has $\partial \widetilde v_i \otimes E_{A_j} = (I_n - \rho(a_j)) \widetilde U + (\rho(u) - I_n) \widetilde A_j$, but 
$\partial \widetilde v_i \otimes E_{B_j} = (I_n - \rho(u)) \widetilde B_j + (\rho(b_j) - I_n) \widetilde U$.
So the sign $(-1)^g$ coming from the orientation discrepancy between the 2-cells $\widetilde A_j$ and $\widetilde B_j$ compensates with the sign $(-1)^g$ coming from the permutation.

Hence we have
$$\tor(X_0, \rho, \mathfrak e_0, \omega^0) =\det(I_n - \rho(u))^{2g+r-1}.$$
Using \cref{eq:TurGlue} we deduce
\begin{equation}
\label{eq:signTor}
\tor(X_1, \rho, \mathfrak e, \omega^1) = \frac{\det(I_n - \rho(u))^{2g+r-2}}{\prod_{j=1}^r \det(I_n-\rho(c_i))}
\end{equation}
and this achieves the proof of \cref{prop:torsion}.

\section{A Selberg trace formula for non-unitary twists}
\label{sec:STF}
We introduce the twisted Bochner--Laplacian on sections of vector bundles over $X$ and use it to derive a trace formula for the corresponding heat operator. In \cref{subsec:vect},
we define the orbifold vector bundles we will make use of, and in \cref{subsec:Laplacians}, the corresponding twisted Bochner--Laplace operators. Then, in \cref{subsec:pretrace}, we write the pre-trace formula for the corresponding heat operator, and we compute the contribution of the identity (\cref{subsec:id}), the hyperbolic (\cref{subsec:hyp}) and elliptic (\cref{subsec:ell}) contributions. We gather all this in \cref{subsec:STF} to write down the Selberg trace formula for (a shift of) the twisted Laplacian. Finally, in \cref{subsec:mult}, we prove a technical result on multiplicities of the eigenvalues of the Laplacians that will be used in \cref{sec:Ruelle} to prove the meromorphicity of the twisted Ruelle zeta function.

\subsection{Orbifold vector bundles over $X$}
\label{subsec:vect}

Recall that $\widetilde{\Gamma}$ denotes the preimage of $\Gamma\subseteq G$ in $\widetilde{G}$ and $\widetilde{\Gamma}\simeq\pi_1(X_1)$. Consider a finite-dimensional, complex representation $\rho:\widetilde{\Gamma}\to\GL(V_\rho)$. Let
$$ E_\rho = V_\rho\times_{\widetilde{\Gamma}}\widetilde{G}\to\widetilde{\Gamma}\backslash\widetilde{G}=X_1 $$
denote the associated flat vector bundle over $X_1$, i.e.
$$ E_\rho = (V_\rho\times\widetilde{G})/_{(\rho(\gamma)v,g)\sim(v,\gamma^{-1}g),\,v\in V_\rho,\gamma\in\widetilde{\Gamma},g\in\widetilde{G}}. $$
We equip $E_\rho$ with a flat connection $\nabla^{E_\chi}$. In general, this vector bundle does not define a vector bundle over $X=X_1/\widetilde{K}$, because both $\widetilde{\Gamma}$ and $\widetilde{K}$ contain the center $\widetilde{Z}$ of $\widetilde{G}$ and $\rho$ is not necessarily trivial on $\widetilde{Z}$. In order to obtain a vector bundle on $X$, we have to twist this construction by a character of $\widetilde{K}$ which is compatible with $\rho$ on $\widetilde{Z}$. We therefore assume that $\rho(u)=e^{-im\pi} \cdot \Id$ for some $m\in\Q$. Note that this is automatically satisfied if $\rho$ is irreducible by \cref{lem:CenterActionIrredPi1X1}.

Let $\tau=\tau_m:\widetilde{K}\to\operatorname{U}(V_\tau)=\operatorname{U}(1)$ denote the character of $\widetilde{K}$ defined in \cref{eq:DefTauM}, then $\rho(u)=\tau(u)^{-1}\cdot\Id_{V_\rho}$. We consider the homogeneous vector bundle 
$$ E_\tau=\widetilde{G}\times_{\widetilde{K}}V_\tau\to \H=\widetilde{G}/\widetilde{K}, $$
i.e.
$$ E_\tau = (\widetilde{G}\times V_\tau)/_{(gk,v)\sim(g,\tau(k)v),\,g\in\widetilde{G},k\in\widetilde{K},v\in V_\tau}. $$
The invariant inner product on $V_\tau$ induces a Hermitian metric and a metric connection $\nabla^{E_\tau}$ on $E_\tau$. By the same reason as above, the bundle $E_\tau$ does not factor through $X=\widetilde{\Gamma}\backslash\widetilde{G}/\widetilde{K}$.

In order to obtain an orbifold vector bundle on $X$ associated with $\rho$, we consider the tensor product representation $V_\tau\otimes V_\rho$ of $\widetilde{K}\times\widetilde{\Gamma}$ and let
$$ E_{\tau,\rho} = \widetilde{G}\times_{\widetilde{K}\times\widetilde{\Gamma}}(V_\tau\otimes V_\rho) \to \widetilde{\Gamma}\backslash\widetilde{G}/\widetilde{K}, $$
where $\widetilde{K}\times\widetilde{\Gamma}$ acts on $\widetilde{G}$ by $g\cdot(k,\gamma)= \gamma^{-1}gk$, so that
$$ E_{\tau,\rho} = (\widetilde{G}\times(V_\tau\otimes V_\rho))/_{(g\cdot(k,\gamma),w)\sim(g,(\tau(k)\otimes\rho(\gamma))w),\,g\in\widetilde{G},k\in\widetilde{K},\gamma\in\widetilde{\Gamma},w\in V_\tau\otimes V_\rho} $$
Note that this action is not faithful, but the elements in $\widetilde{K}\times\widetilde{\Gamma}$ acting trivially on $\widetilde{G}$ are of the form $(z,z)$, $z\in\widetilde{Z}$, and hence also act trivially on $V_\tau\otimes V_\rho$ by construction. We further remark that this does not define a topological vector bundle in the usual sense but only an orbifold vector bundle, because the fiber over a singular point is not a vector space but rather a quotient of a vector space by a finite group action. However, the orbisurface $X$ has a finite cover which is a manifold, and the orbifold vector bundle is induced by a topological vector bundle (in the usual sense) on the cover. We therefore abuse notation and talk about connections, metrics and differential operators on the bundle, meaning the corresponding objects on the finite cover acting on sections invariant under the finite group of Deck transformations. A more thorough discussion of orbifold vector bundles can be found in the work of Shen and Yu~\cite{SY22}.

Smooth sections of the bundle $E_{\tau,\rho}$ can be identified with smooth functions $f:\widetilde{G}\to V_\tau\otimes V_\rho$, such that
$$ f(\gamma gk) = \left[\tau(k)^{-1}\otimes\rho(\gamma)\right]f(g) \qquad (g\in\widetilde{G},k\in\widetilde{K},\gamma\in\widetilde{\Gamma}). $$

Note that if $\rho$ is trivial and $X$ has no singularities, the bundle $E_{\tau_m,1}\to X$ for $m\in2\Z$ is the bundle of Fourier modes of degree $m$, whose sections are identified with smooth functions $f : X_1 \to \C$ such that
$$ f(x,v) = f (x,R_\theta v)e^{i\pi m\theta} \qquad \mbox{for all }x\in X,v\in T_xX,|v|=1, $$
where $R_\theta$ is the rotation of angle $\theta$ in $T_xX$.

We now define a connection $\nabla^{E_{\tau,\rho}}$ on $E_{\tau,\rho}$ in terms of its associated covariant derivative. For this we identify vector fields $\mathfrak X\in C^\infty(X,TX)$ on $X$ with smooth functions $\mathfrak X:\widetilde{G}\to\mathfrak{p}$ such that $\mathfrak X(\gamma gk)=\Ad(k)|_{\mathfrak{p}}^{-1}\mathfrak X(g)$. Then,
$$ \nabla^{E_{\tau,\rho}}_{\mathfrak X}f(g) = \left.\frac{d}{dt}\right|_{t=0}f(g\exp(t\mathfrak X(g))) \qquad (f\in C^\infty(X,E_{\tau,\rho}),\mathfrak X\in C^\infty(X,TX)) $$
defines a covariant derivative and hence a connection on $E_{\tau,\rho}$. Note that in general there is no Hermitian metric on $E_{\tau,\rho}$ compatible with the connection, since the representation $\rho$ is not necessarily unitary.

\subsection{The twisted Bochner--Laplacian}
\label{subsec:Laplacians}

For a complex vector bundle $E\to X$ with covariant derivative $\nabla^E$, the second covariant derivative $(\nabla^E)^2$ is defined by
$$ (\nabla^E)^2_{\mathfrak X,\mathfrak X'} = \nabla_{\mathfrak X}^E\nabla_{\mathfrak X'}^E-\nabla^E_{\nabla_{\mathfrak X}^{\operatorname{LC}}\mathfrak X'} \qquad (\mathfrak X,\mathfrak X'\in C^\infty(X,TX)), $$
where $\nabla^{\operatorname{LC}}$ is the Levi--Civita connection on $TX$.
The negative of the trace of the second covariant derivative is the corresponding Bochner--Laplacian:
$$ \Delta_E = -\tr\left((\nabla^E)^2\right). $$

For $\rho$ and $\tau$ as in the previous section, the twisted Bochner--Laplacian $\Delta_{\tau,\rho}^\sharp$ is defined as the Bochner--Laplacian of the bundle $E_{\tau,\rho}$:
$$ \Delta_{\tau,\rho}^\sharp = \Delta_{E_{\tau,\rho}} = -\tr\left((\nabla^{E_{\tau,\rho}})^2\right).$$
Note that the orbisurface $X$ has a finite cover, which is a manifold. Therefore, the spectral theory of $\Delta_{\tau,\rho}^\sharp$ on $X$ is the same as the spectral theory of its lift to the finite cover, acting on sections invariants under the action of the finite group of Deck transformations. Hence, by \cite[Theorem 4.3]{Mk}, we have the following properties (see also for the manifold case the work of M\"{u}ller~\cite{M1}, and for the orbifold case the work of Fedosova~\cite{fedell} and Shen~\cite[Section 7]{Shen2020}):
\begin{enumerate}[(1)]
\item If we choose a Hermitian metric on $E_{\tau,\rho}$, then $\Delta_{\tau,\rho}^\sharp$ acts in $L^2(X,E_{\tau,\rho})$ with domain $C^\infty(X,E_{\tau,\rho})$. However, it is not a formally self-adjoint operator in general. Note that while the inner product on $L^2(X,E_{\tau,\rho})$ depends on the chosen inner product on $V_\rho$, the space $L^2(X,E_{\tau,\rho})$ as a topological vector space does not, thanks to the compactness of $X$.
\item $\Delta_{\tau,\rho}^\sharp$ is an elliptic second order differential operator with purely discrete spectrum $\spec(\Delta_{\tau,\rho}^\sharp)\subseteq\C$, consisting of generalized eigenvalues. The spectrum is contained in a translate of a positive cone in $\C$. (Here, a positive cone is a cone whose closure is contained in $\{z\in\C:\RE z>0\}\cup\{0\}$.) The generalized eigenspaces
$$ \{f\in L^2(X,E_{\tau,\rho}):(\Delta_{\tau,\rho}^\sharp-\mu\Id)^Nf=0\mbox{ for some }N\} $$
are finite-dimensional, contained in $C^\infty(X,E_{\tau,\rho})$, and their direct sum, as $\mu$ runs through $\spec(\Delta_{\tau,\rho}^\sharp)$, is dense in $L^2(X,E_{\tau,\rho})$.
\end{enumerate}

\subsection{The pre-trace formula}
\label{subsec:pretrace}

The heat operator $e^{-t\Delta_{\tau_m,\rho}^\sharp}$ corresponding to the twisted Laplacian $\Delta_{\tau_m,\rho}^\sharp$ is of trace class (follows e.g. from the Weyl Law in \cite[Lemma 2.2]{M1}) and by \cite[equation (5.6)]{M1} its integral kernel is the smooth function
$$ H_t^{\tau_m,\rho}(g_1,g_2) = \sum_{[\gamma]\subseteq\widetilde{\Gamma}/\widetilde{Z}} \rho(\gamma)\otimes H_t^{\tau_m}(g_2^{-1}\gamma g_1) \qquad (g_1,g_2\in\widetilde{G}), $$
where the summation is over the conjugacy classes of $\widetilde{\Gamma}/\widetilde{Z}\simeq\Gamma$ and the heat kernel of the Bochner--Laplace operator on $E_{\tau_m}\to\H$ is written as $(g_1,g_2)\mapsto H_t^{\tau_m}(g_2^{-1}g_1)$ with
$$ H_t^{\tau_m}\in(C^\infty(\widetilde{G})\otimes\End(V_{\tau_m}))^{\widetilde{K}\times\widetilde{K}}. $$
Here, $\widetilde{K}\times\widetilde{K}$ acts on $C^\infty(\widetilde{G})$ by left and right translation, on $\End(V_{\tau_m})$ by left and right composition with $\tau_m$, and on their tensor product by the tensor product of these actions. Note that the center $\widetilde Z$ of $\widetilde G$ lies in $\widetilde K$, hence each summand is independent on the choice of $\gamma$ relatively to $\widetilde Z$. The trace of $e^{-t\Delta_{\tau_m,\rho}^\sharp}$ can be computed in two different ways, by summing over the generalized eigenvalues of $\Delta_{\tau_m,\rho}^\sharp$ (the spectral side) and by integrating the heat kernel along the diagonal in $\widetilde{\Gamma}\backslash\widetilde{G}$ (the geometric side). For the spectral side we denote for an eigenvalue $\mu\in\spec(\Delta_{\tau_m,\rho}^\sharp)$ its algebraic multiplicity by
$$ \mult(\mu;\Delta_{\tau_m,\rho}^\sharp) = \dim\{f\in L^2(X,E_{\tau_m,\rho}):(\Delta_{\tau_m,\rho}^\sharp-\mu\Id)^Nf=0\mbox{ for some }N\}. $$
Note that $\mult(\mu;\Delta_{\tau_m,\rho}^\sharp)$ is independent of the chosen metric on $E_\rho$. For the geometric side, we use the standard arguments grouping conjugacy classes in $\widetilde{\Gamma}/\widetilde{Z}$ (see e.g., \cite[Section 2]{Wa}). This leads to the following pre-trace formula which was made rigorous by M\"{u}ller~\cite[Proposition 5.1]{M1} for manifolds and generalized to orbifolds by Shen~\cite[Theorem 7.1]{Shen2020}:
\begin{multline}
	\sum_{\mu\in\spec(\Delta_{\tau_m,\rho}^\sharp)}\mult(\mu;\Delta_{\tau_m,\rho}^\sharp)e^{-t\mu}\\
	= \sum_{[\gamma]\subseteq\widetilde{\Gamma}/\widetilde{Z}}\Vol(\widetilde{\Gamma}_\gamma\backslash\widetilde{G}_\gamma)\tr{\rho(\gamma)}\int_{\widetilde{G}_\gamma\backslash\widetilde{G}}\tr H_t^{\tau_m}(g^{-1}\gamma g)\,d\dot{g},\label{eq:PreTrace}
\end{multline}
where $\widetilde{G}_\gamma$ resp. $\widetilde{\Gamma}_\gamma$ denotes the centralizer of $\gamma\in\Gamma$ in $\widetilde{G}$ resp. $\widetilde{\Gamma}$, and the summation on the right hand side is over all conjugacy classes $[\gamma]$ of elements in $\widetilde{\Gamma}/\widetilde{Z}\simeq\Gamma$.

Since $\widetilde{\Gamma}/\widetilde{Z}\simeq\Gamma$ consists of the identity element $e$, hyperbolic and elliptic elements, we can compute the three contributions to the right-hand side in \cref{eq:PreTrace} separately.

\begin{rmrk}
	Since $m\in\Q$ we could just as well formulate everything in terms of a finite covering group $G_1$ of $G$ which is a connected, semisimple group with finite center. This formulation allows us to use the Harish-Chandra $L^q$-Schwartz space $\mathcal{C}^q(G_1)$ (see e.g. \cite[p. 161--162]{BM} for its definition). By \cite[Lemma 2.3 and Proposition 2.4]{BM}, the heat kernel $H_t^{\tau_m}$ is contained in $\mathcal{C}^q(G_1)\otimes\End(V_{\tau_m})$ for any $q>0$, and hence it is an admissible function in the sense of Gangolli~\cite[p. 407]{gangolli1977}. We can therefore apply the distribution character of a unitary representation of $G_1$ to $H_t^{\tau_m}$ in \cref{subsec:id}, \cref{subsec:hyp} and \cref{subsec:ell}.
\end{rmrk}

\subsection{The identity contribution}
\label{subsec:id}

The contribution of $\gamma=e$ to the right-hand side of \cref{eq:PreTrace} is
$$ \Vol(X)\dim(V_\rho)H_t^{\tau_m}(e).$$
We can employ the Plancherel formula for $L^2(\H,E_{\tau_m})$ from \cite[Lemma 5]{Hof94}. (Note that although \cite[Lemma 5]{Hof94} requires $|m|\leq1$, the formula only depends on $m+2\Z$, so it holds for arbitrary $m\in\Z$.) This gives
\begin{multline*}
	H_t^{\tau_m}(e) = \frac{1}{4\pi}\Bigg[\int_\R \Theta_{\sigma,\lambda}(H_t^{\tau_m})\frac{\lambda\sinh(2\pi\lambda)}{\cosh(2\pi\lambda)+\cos(\pi m)}\,d\lambda\\
	+ \sum_{n\equiv m+1\mod2}|n|\Theta_{n+\sign(n)}(H_t^{\tau_m})\Bigg],
\end{multline*}
where $\sigma$ is the restriction of $\tau_m$ to $\widetilde{M}$. By \cite[p. 2634]{MPf} we find that
\begin{equation}
	\Theta_{\sigma,\lambda}(H_t^{\tau_m}) = e^{-t(\lambda^2+\frac{1}{4})}\label{eq:PSCharacterHeatKernel}
\end{equation}
and
\begin{equation}
	\Theta_{m'}(H_t^{\tau_m}) = \begin{cases}e^{-\frac{|m'|}{2}(1-\frac{|m'|}{2})t}&\mbox{for $m=m'+2\sign(m')\ell$, $\ell\in\Z_{\geq0}$,}\\0&\mbox{else.}\end{cases}\label{eq:DSCharacterHeatKernel}
\end{equation}
Note that, in contrast to the case of for instance higher-dimensional hyperbolic manifolds, the expression for $\Theta_{\sigma,\lambda}(H_t^{\tau_m})$ is independent of $\sigma$ since $\widetilde{M}$ is discrete. Write $n+\sign(n)=\sign(m)(|m|-(\ell-1))$, $1\leq\ell<|m|$ odd, then the identity contribution becomes
\begin{multline*}
	\frac{\Vol(X)\dim(V_\rho)}{4\pi}\Bigg[\int_\R e^{-t(\lambda^2+\frac{1}{4})}\frac{\lambda\sinh(2\pi\lambda)}{\cosh(2\pi\lambda)+\cos(\pi m)}\,d\lambda\\
	+ \sum_{\substack{1\leq\ell<|m|\\\ell\textup{ odd}}}(|m|-\ell)e^{\frac{(|m|-\ell+1)(|m|-\ell-1)}{4}t}\Bigg].
\end{multline*}

\subsection{The hyperbolic contribution}
\label{subsec:hyp}

For a conjugacy class $[\gamma]\subseteq\widetilde{\Gamma}/\widetilde{Z}$ of hyperbolic elements we choose a representative $\gamma\in\widetilde{\Gamma}$ which is conjugate in $\widetilde{G}$ to $\widetilde{a}_{\ell(\gamma)}$, instead of $z\widetilde{a}_{\ell(\gamma)}$, $z\in\widetilde{Z}$.
Here, $\ell(\gamma)>0$ is the length of the geodesic corresponding to $[\gamma]$. Then, by \cite[Lemma 2]{Hof94} and \cite[Section 6]{Wa}, the contribution of $[\gamma]$ to the right-hand side of \cref{eq:PreTrace} is given by
$$ \tr\rho(\gamma)\frac{\ell(\gamma)}{n_\Gamma(\gamma)D(\gamma)}\frac{1}{2\pi}\int_\R \Theta_{\sigma,\lambda}(H_t^{\tau_m})e^{-i\ell(\gamma)\lambda}\,d\lambda, $$
where $\sigma$ is the restriction of $\tau_m$ to $\widetilde{M}$, $n_{\Gamma}(\gamma)$ is the multiplicity of the (not necessarily prime) geodesic $[\gamma]$ and $D(\gamma) = 2 \sinh(\frac {\ell(\gamma)} 2)$ is the Weyl denominator. 
By \cref{eq:PSCharacterHeatKernel}, the total hyperbolic contribution becomes
$$ \sum_{[\gamma]\textup{ hyp.}}\tr\rho(\gamma)\frac{\ell(\gamma)}{n_\Gamma(\gamma)D(\gamma)}\frac{1}{2\pi}\int_\R e^{-t(\lambda^2+\frac{1}{4})}e^{-i\ell(\gamma)\lambda}\,d\lambda. $$
Computing the integral,
we obtain
$$ \frac{e^{-\frac{t}{4}}}{2\sqrt{4\pi t}}\sum_{[\gamma]\textup{ hyp.}}\frac{\ell(\gamma)\tr\rho(\gamma)}{n_\Gamma(\gamma)\sinh(\frac{\ell(\gamma)}{2})}e^{-\frac{\ell(\gamma)^2}{4t}}. $$
We remark that this sum converges absolutely by the same arguments as in \cite[Proposition 4.1.1]{FS}.

\subsection{The elliptic contribution}
\label{subsec:ell}

For a conjugacy class $[\gamma]\subseteq\widetilde{\Gamma}/\widetilde{Z}$ of elliptic elements, we choose a representative $\gamma\in\widetilde{\Gamma}$ which is conjugate in $\widetilde{G}$ to $\widetilde{k}_{\theta(\gamma)}$ with $\theta(\gamma)\in(0,\pi)$. Let $M(\gamma)$ denote the order of the stabilizer $\widetilde{\Gamma}_\gamma/\widetilde{Z}$ of $\gamma$ in $\widetilde{\Gamma}/\widetilde{Z}$, then
$$ \Vol(\widetilde{\Gamma}_\gamma\backslash\widetilde{G}_\gamma)=\Vol(\Gamma_\gamma\backslash G_\gamma)=\Vol(\langle k_{\theta(\gamma)}\rangle\backslash\PSO(2))=\frac{1}{M(\gamma)} $$
and hence, by \cite[Lemma 4]{Hof94} (which again holds for all $m\in\Z$ since the formula only depends on $m+2\Z$), the contribution of $[\gamma]$ to the pre-trace formula \cref{eq:PreTrace} is given by
\begin{multline*}
	\frac{\tr\rho(\gamma)}{M(\gamma)}\cdot\frac{1}{2i\sin(\theta(\gamma))}\Bigg[\frac{i}{2}\int_\R\Theta_{\sigma,\lambda}(H_t^{\tau_m})\frac{\cosh(2(\pi-\theta(\gamma))\lambda)+e^{i\pi m}\cosh(2\theta(\gamma)\lambda)}{\cosh2\pi\lambda+\cos\pi m}\,d\lambda\\
	-\delta_{(|m|-1)\!\!\!\!\!\mod2}\frac{\Theta_1(H_t^{\tau_m})-\Theta_{-1}(H_t^{\tau_m})}{2}\\
	-\sum_{\substack{n\equiv m+1\!\!\!\!\mod2\\n\neq0}}\sign(n)e^{in\theta(\gamma)}\Theta_{n+\sign(n)}(H_t^{\tau_m})\Bigg],
\end{multline*}
where $\delta_{(|m|-1)\!\!\mod2}$ is meant to be $1$ for $|m|-1\equiv0\mod2$ and $0$ otherwise. Using \cref{eq:PSCharacterHeatKernel} and \cref{eq:DSCharacterHeatKernel}, the elliptic contribution becomes
\begin{multline*}
	\sum_{[\gamma]\textup{ ell.}} \frac{\tr\rho(\gamma)}{4M(\gamma)\sin(\theta(\gamma))}\Bigg[\int_\R e^{-t(\lambda^2+\frac{1}{4})}\frac{\cosh(2(\pi-\theta(\gamma))\lambda)+e^{i\pi m}\cosh(2\theta(\gamma)\lambda)}{\cosh(2\pi\lambda)+\cos(\pi m)}\,d\lambda\\
	+2i\sign(m)\sum_{\substack{1\leq\ell<|m|\\\ell\textup{ odd}}}e^{i\sign(m)(|m|-\ell)\theta(\gamma)}e^{\frac{(|m|-\ell+1)(|m|-\ell-1)}{4}t}\Bigg].
\end{multline*}
Note that this sum is always finite since $\Gamma$ only contains finitely many conjugacy classes of elliptic elements.

\subsection{The trace formula}
\label{subsec:STF}
Adding the identity, the hyperbolic and the elliptic contribution finally shows the trace formula. For the statement let
$$ A_{\tau_m,\rho}^\sharp=\Delta_{\tau_m,\rho}^\sharp-\frac{1}{4}. $$

\begin{theorem}[Selberg trace formula with non-unitary twist]\label{thm:TraceFormula}
	Let $\rho:\widetilde{\Gamma}\to\GL(V_\rho)$ be a finite-dimensional complex representation with $\rho(u)=e^{-im\pi}$, $m\in\R$. Then, for every $t>0$ we have
	\begin{align*}
		& \Tr(e^{-tA_{\tau_m,\rho}^\sharp})\\
		& \quad=\frac{\Vol(X)\dim(V_\rho)}{4\pi}\Bigg[\int_\R e^{-t\lambda^2}\frac{\lambda\sinh(2\pi\lambda)}{\cosh(2\pi\lambda)+\cos(\pi m)}\,d\lambda\\
		& \hspace{7cm}+ \sum_{\substack{1\leq\ell<|m|\\\ell\textup{ odd}}}(|m|-\ell)e^{(\frac{|m|-\ell}{2})^2t}\Bigg].\\
		& \qquad+\frac{1}{2\sqrt{4\pi t}}\sum_{[\gamma]\textup{ hyp.}}\frac{\ell(\gamma)\tr\rho(\gamma)}{n_\Gamma(\gamma)\sinh(\frac{\ell(\gamma)}{2})}e^{-\frac{\ell(\gamma)^2}{4t}}\\
		& \qquad+\sum_{[\gamma]\textup{ ell.}} \frac{\tr\rho(\gamma)}{4M(\gamma)\sin(\theta(\gamma))}\Bigg[\int_\R e^{-t\lambda^2}\frac{\cosh(2(\pi-\theta(\gamma))\lambda)+e^{i\pi m}\cosh(2\theta(\gamma)\lambda)}{\cosh(2\pi\lambda)+\cos(\pi m)}\,d\lambda\\
		& \hspace{4cm}+2i\sign(m)\sum_{\substack{1\leq\ell<|m|\\\ell\textup{ odd}}}e^{i\sign(m)(|m|-\ell)\theta(\gamma)}e^{(\frac{|m|-\ell}{2})^2t}\Bigg].
	\end{align*}
	Here, the summation is over the conjugacy classes $[\gamma]$ of hyperbolic, resp. elliptic, elements in $\widetilde{\Gamma}/\widetilde{Z}\simeq\Gamma$ and the representative $\gamma\in\widetilde{\Gamma}$ is chosen such that it is conjugate in $\widetilde{G}$ to $\widetilde{a}_{\ell(\gamma)}$ ($\ell(\gamma)>0$), resp. $\widetilde{k}_{\theta(\gamma)}$ ($\theta(\gamma)\in(0,\pi)$).
\end{theorem}

\subsection{An application to the determination of multiplicities}
\label{subsec:mult}
In order to conclude that the residues of the logarithmic derivative of the Selberg zeta function are integers, we now deduce from the trace formula an expression for the difference of the multiplicities $\mult(\mu;\Delta_{\tau_m,\rho}^\sharp)$ for $m$ and $m+2$. For this, it is more convenient to work with $A_{\tau_m,\rho}^\sharp$. For $\mu\in\C$ we let
$$ L^2(X,E_{\tau_m,\rho})_\mu = \{f\in L^2(X,E_{\tau_m,\rho}):(A_{\tau_m,\rho}^\sharp-\mu\Id)^Nf=0\mbox{ for some }N\} $$
denote the corresponding generalized eigenspace of $A_{\tau_m,\rho}^\sharp$ and
$$ \mult(\mu;A_{\tau_m,\rho}^\sharp) = \dim L^2(X,E_{\tau_m,\rho})_\mu $$
its multiplicity.

Define $\kappa,x_\pm\in\sl(2,\C)$ by
$$ \kappa = E-F, \qquad x_\pm = \frac{1}{2}(H\pm i(E+F)), $$
then
$$ [\kappa,x_\pm] = \pm2ix_\pm, \qquad [x_+,x_-] = -i\kappa. $$
We remark that the complexification $\sl(2,\C)$ of the Lie algebra $\sl(2,\R)$ of $G$ has to be used here since the adjoint action of $\kappa$ on $\sl(2,\R)$ is not diagonalizable. The Casimir element $\Omega$ for $\sl(2,\C)$ can be expressed in terms of the two bases $H,E,F$ and $\kappa,x_+,x_-$ as
$$ \Omega = \frac{1}{4}(H^2+2EF+2FE) = -\frac{1}{4}(\kappa^2-2x_+x_--2x_-x_+). $$
Moreover, we can write
\begin{equation}
	\Omega = -\frac{1}{4}\kappa^2-\frac{1}{2}i\kappa+x_-x_+ = -\frac{1}{4}\kappa^2+\frac{1}{2}i\kappa+x_+x_-.\label{eq:AlternativeExpressionCasimir}
\end{equation}

For $u\in C^\infty(\H,E_{\tau_m})=(C^\infty(G)\otimes V_{\tau_m})^K$ we define $\partial_\pm u(g)=d\pi_R(x_\pm)u(g)$, where $d\pi_R$ denotes the Lie algebra representation corresponding to the  right regular representation $\pi_R$ of $G$ on $C^\infty(G)$. Then $[\kappa,x_\pm]=\pm 2ix_\pm$ implies
$$ \partial_\pm:C^\infty(\H,E_{\tau_m})\to C^\infty(\H,E_{\tau_{m\mp2}}). $$
Further, by \cref{eq:AlternativeExpressionCasimir} we find that for $u\in C^\infty(\H,E_{\tau_m})$:
\begin{align*}
	\partial_+\partial_-u &= d\pi_R(\Omega)+\frac{1}{4}d\pi_R(\kappa)^2-\frac{1}{2}id\pi_R(\kappa) = d\pi_L(\Omega)-\frac{m^2}{4}-\frac{m}{2},\\
	\partial_-\partial_+u &= d\pi_R(\Omega)+\frac{1}{4}d\pi_R(\kappa)^2+\frac{1}{2}id\pi_R(\kappa) = d\pi_L(\Omega)-\frac{m^2}{4}+\frac{m}{2},
\end{align*}
where $d\pi_L$ is the Lie algebra representation corresponding to the left regular representation $\pi_L$. Note that $d\pi_L(\Omega)=-\widetilde{\Delta}_{\tau_m,\rho}^\sharp=-\widetilde{A}_{\tau_m,\rho}^\sharp-\frac{1}{4}$, where $\widetilde{\Delta}_{\tau_m,\rho}^\sharp$ resp. $\widetilde{A}_{\tau_m,\rho}^\sharp$ denotes the lift of $\Delta_{\tau_m,\rho}^\sharp$ resp. $A_{\tau_m,\rho}^\sharp$ to $\H$. Since $\partial_+$ and $\partial_-$ are left-invariant, they induce operators on $C^\infty(X,E_{\tau_m,\rho})$ which, by the previous considerations, satisfy
\begin{equation}
	\partial_+\partial_-u = -A_{\tau_m,\rho}^\sharp-\frac{(m+1)^2}{4}, \qquad \partial_-\partial_+u = -A_{\tau_m,\rho}^\sharp-\frac{(m-1)^2}{4}.\label{eq:PartialPlusMinusVsLaplacian}
\end{equation}

\begin{lem}\label{lem:MultiplicitiesModM}
	For $\mu\neq-\frac{(m+1)^2}{4}$, we have:
	$$ \mult(\mu;A_{\tau_m,\rho}^\sharp) = \mult(\mu;A_{\tau_{m+2},\rho}^\sharp). $$
\end{lem}

\begin{proof}
	Since $\partial_\pm$ is defined in terms of the right regular representation and $A_{\tau_m,\rho}^\sharp$ in terms of the left regular representation, we have
	$$ \partial_\pm\circ A_{\tau_m,\rho}^\sharp = A_{\tau_{m\mp2},\rho}^\sharp\circ\partial_\pm. $$
	This implies
	$$ \partial_\pm:L^2(X,E_{\tau_m,\rho})_\mu\to L^2(X,E_{\tau_{m\mp2},\rho})_\mu. $$
	By \cref{eq:PartialPlusMinusVsLaplacian}, the map
	$$ \partial_+:L^2(X,E_{\tau_m,\rho})_\mu\to L^2(X,E_{\tau_{m-2},\rho})_\mu $$
	is an isomorphism whenever $\mu\neq-\frac{(m-1)^2}{4}$, and the map
	$$ \partial_-:L^2(X,E_{\tau_m,\rho})_\mu\to L^2(X,E_{\tau_{m+2},\rho})_\mu $$
	is an isomorphism whenever $\mu\neq-\frac{(m+1)^2}{4}$.
\end{proof}

\begin{theorem}\label{thm:DimensionEigenspaces}
	For all $m\in\R$ with $\rho(u)=e^{-im\pi}\cdot\Id_{V_\rho}$, we have
	\begin{multline*}
		\mult(-\tfrac{(m+1)^2}{4};A_{\tau_{m+2},\rho}^\sharp)-\mult(-\tfrac{(m+1)^2}{4};A_{\tau_m,\rho}^\sharp)\\
		= \frac{\Vol(X)\dim(V_\rho)}{4\pi}(m+1) +\sum_{[\gamma]\textup{ ell.}}\frac{i\tr\rho(\gamma)}{2M(\gamma)\sin(\theta(\gamma))}e^{i(m+1)\theta(\gamma)}.
	\end{multline*}
	In particular, the right-hand side is an integer.
\end{theorem}

\begin{proof}
	We subtract the trace formula in \cref{thm:TraceFormula} for $m$ from the one for $m+2$. By \cref{lem:MultiplicitiesModM}, the left-hand side becomes
	$$ e^{\frac{(m+1)^2}{4}t}\big(\mult(-\tfrac{(m+1)^2}{4};A_{\tau_{m+2},\rho}^\sharp)-\mult(-\tfrac{(m+1)^2}{4};A_{\tau_m,\rho}^\sharp)\big). $$
	Carefully comparing terms on the right-hand side for $m$ and $m+2$, we find
	\begin{equation*}
		\frac{\Vol(X)\dim(V_\rho)}{4\pi}(m+1)e^{\frac{(m+1)^2}{4}t}+\sum_{[\gamma]\textup{ ell.}}\frac{i\tr\rho(\gamma)}{2M(\gamma)\sin(\theta(\gamma))}e^{i(m+1)\theta(\gamma)}e^{\frac{(m+1)^2}{4}t}.\qedhere
	\end{equation*}
\end{proof}

\section{Twisted Ruelle and Selberg zeta functions}
\label{sec:Ruelle}
We establish the meromorphic continuation for the Selberg zeta function in \cref{subsec:mero}, and a functional equation in \cref{subsec:FE}. Then, we compute the behavior of the Ruelle zeta function at zero in \cref{subsec:Fried}, and prove \cref{theo:mainFried}.

\subsection{Meromorphic continuation}
\label{subsec:mero}

Every closed oriented geodesic $\gamma$ on $X$ lifts canonically to a homotopy class in $\widetilde{\Gamma}=\pi_1(X_1)$, still denoted by $\gamma$. The conjugacy class $[\gamma]$ in $\widetilde{\Gamma}/\widetilde{Z}\simeq\Gamma=\pi_1(X)$ consists of hyperbolic elements and conversely every such conjugacy class of hyperbolic elements contains exactly one representative of a closed geodesic.

A closed geodesic $\gamma$ on $X$ is called \emph{prime} if it cannot be written as a multiple of a shorter geodesic. 

Let $\rho:\widetilde{\Gamma}=\pi_1(X_1)\to\GL(V_\rho)$ be a finite-dimensional, complex representation. For $s\in\C$, we define the twisted Selberg zeta function
$$ Z(s;\rho) = \prod_{\gamma\text{ prime}}\prod_{k=0}^\infty\det\left(\Id-\rho(\gamma)e^{-(s+k)\ell(\gamma)}\right) $$
and the twisted Ruelle zeta function
$$ R(s;\rho) = \prod_{\gamma\text{ prime}}\det\left(\Id-\rho(\gamma)e^{-s\ell(\gamma)}\right), $$
where the products run over the prime closed geodesics in $X$.

It is shown in \cite[Theorem 3.1]{fedosova2020meromorphic} that $Z(s;\rho)$, and hence $R(s;\rho)$, converges for $s$ in some right half plane of $\C$ and defines a holomorphic function on this half plane (see also \cite[Section 1.2]{Wo} for the torsion-free case). We note that the logarithmic derivative 
$L(s;\rho)$ of $Z(s;\rho)$ is given by
$$ L(s;\rho):= \frac{d}{ds}\log Z(s;\rho) = \sum_\gamma\frac{\ell(\gamma)\tr(\rho(\gamma))}{2n_\Gamma(\gamma)\sinh(\frac{\ell(\gamma)}{2})}e^{-(s-\frac{1}{2})\ell(\gamma)}, $$
where the summation is over \emph{all} closed geodesics. 

Moreover, we have
\begin{equation}
\label{eq:RuelSel} R(s;\rho) = \frac{Z(s;\rho)}{Z(s+1;\rho)}.
\end{equation}

To show meromorphic continuation of $Z(s;\rho)$ and $R(s;\rho)$ to $s\in\C$, we let $s_1=s\in\C$ and $s_2\in\C$ with $\RE\big((s_1-\frac{1}{2})^2\big),\RE\big((s_2-\frac{1}{2})^2\big)$ sufficiently large and insert the trace formula from \cref{thm:TraceFormula} into the resolvent identity
\begin{multline*}
	\tr\left(\left(A_{\tau_m,\rho}^\sharp+(s_1-\tfrac{1}{2})^2\right)^{-1}-\left(A_{\tau_m,\rho}^\sharp+(s_2-\tfrac{1}{2})^2\right)^{-1}\right)\\
	= \int_0^\infty \left(e^{-t(s_1-\frac{1}{2})^2}-e^{-t(s_2-\frac{1}{2})^2}\right)\tr(e^{-tA_{\tau_m,\rho}^\sharp})\,dt,
\end{multline*}
which holds by the same arguments as in \cite[Section 4.2]{FS} (especially the proof of \cite[Lemma 4.2.2]{FS} is identical). Note that the individual resolvents on the left hand side are not of trace class, but their difference is (see \cite[p.~51]{FS} for details). In this way, we obtain
\begin{multline}
	\tr\left(\left(A_{\tau_m,\rho}^\sharp+(s_1-\tfrac{1}{2})^2\right)^{-1}-\left(A_{\tau_m,\rho}^\sharp+(s_2-\tfrac{1}{2})^2\right)^{-1}\right)\\
	= I(s_1,s_2;\rho)+H(s_1,s_2;\rho)+E(s_1,s_2;\rho)\label{eq:ResolventIdentityWithIdHypEllContributions},
\end{multline}
with the obvious meaning of the identity contribution $I(s_1,s_2;\rho)$, the hyperbolic contribution $H(s_1,s_2;\rho)$ and the elliptic contribution $E(s_1,s_2;\rho)$. For the computation of these three terms we choose $|m|\leq1$ which is possible since $m$ is determined by $\rho(u)=e^{-im\pi}$ and hence only determined by $\rho$ modulo $2\Z$. Then, all contributions from the (relative) discrete series vanish.

We first compute the identity contribution
\begin{multline}
		I(s_1,s_2;\rho) = \frac{\Vol(X)\dim(V_\rho)}{4\pi}\int_0^\infty\left(e^{-t(s_1-\frac{1}{2})^2}-e^{-t(s_2-\frac{1}{2})^2}\right)\\
		\times\int_\R e^{-t\lambda^2}\frac{\lambda\sinh(2\pi\lambda)}{\cosh(2\pi\lambda)+\cos(\pi m)}\,d\lambda\,dt.
\end{multline}
Interchanging the integrals shows that
\begin{equation*}
	= \frac{\Vol(X)\dim(V_\rho)}{4\pi}\int_\R\frac{\lambda\sinh(2\pi\lambda)}{\cosh(2\pi\lambda)+\cos(\pi m)}\left(\frac{1}{\lambda^2+(s_1-\frac{1}{2})^2}-\frac{1}{\lambda^2+(s_2-\frac{1}{2})^2}\right)\,d\lambda.
\end{equation*}
This integral is computed in \cite[Chapter 10, Lemma 2.4]{He2}:
\begin{equation}
	= \frac{\Vol(X)\dim(V_\rho)}{4\pi}\sum_{n=0}^\infty\Bigg[\frac{1}{n+\frac{m}{2}+s_1}+\frac{1}{n-\frac{m}{2}+s_1}-\frac{1}{n+\frac{m}{2}+s_2}-\frac{1}{n-\frac{m}{2}+s_2}\Bigg]\label{eq:IdentityFinal}.
\end{equation}

The hyperbolic contribution
\begin{align*}
	H(s_1,s_2;\rho) ={}& \int_0^\infty\left(e^{-t(s_1-\frac{1}{2})^2}-e^{-t(s_2-\frac{1}{2})^2}\right)\\
	&\hspace{1.5cm}\times\frac{1}{2\sqrt{4\pi t}}\sum_{[\gamma]\textup{ hyp.}}\frac{\ell(\gamma)\tr(\rho(\gamma))}{n_{\Gamma}(\gamma)\sinh(\frac{\ell(\gamma)}{2})}e^{-\frac{\ell(\gamma)^2}{4t}}\,dt
\end{align*}
was calculated in \cite[Proposition 4.2.4 \& 4.2.5]{FS} to be
$$ = \frac{L(s_1;\rho)}{2(s_1-\frac{1}{2})}-\frac{L(s_2;\rho)}{2(s_2-\frac{1}{2})}. $$

Finally, the elliptic contribution is given by
\begin{multline*}
	E(s_1,s_2;\rho) = \sum_{[\gamma]\textup{ ell.}} \frac{\tr\rho(\gamma)}{4M(\gamma)\sin(\theta(\gamma))}\int_0^\infty\left(e^{-t(s_1-\frac{1}{2})^2}-e^{-t(s_2-\frac{1}{2})^2}\right)\\
	\times\int_\R e^{-t\lambda^2}\frac{\cosh(2(\pi-\theta(\gamma))\lambda)+e^{i\pi m}\cosh(2\theta(\gamma)\lambda)}{\cosh(2\pi\lambda)+\cos(\pi m)}\,d\lambda\,dt.
\end{multline*}
Note that summation and integration can be interchanged since the sum is finite. Interchanging the two integrals and computing the integral over $t$ gives
\begin{multline*}
	= \sum_{[\gamma]\textup{ ell.}} \frac{\tr\rho(\gamma)}{4M(\gamma)\sin(\theta(\gamma))}\int_\R\left(\frac{1}{\lambda^2+(s_1-\frac{1}{2})^2}-\frac{1}{\lambda^2+(s_2-\frac{1}{2})^2}\right)\\
	\times\frac{\cosh(2(\pi-\theta(\gamma))\lambda)+e^{i\pi m}\cosh(2\theta(\gamma)\lambda)}{\cosh(2\pi\lambda)+\cos(\pi m)}\,d\lambda.
\end{multline*}
The integral over $\lambda$ is computed in \cite[Chapter 10, Lemma 2.7]{He2}:
\begin{multline}
	= \sum_{[\gamma]\textup{ ell.}} \frac{\tr\rho(\gamma)}{4M(\gamma)\sin(\theta(\gamma))}\Bigg[\frac{i}{s_1-\frac{1}{2}}\sum_{n=0}^\infty\Bigg(\frac{e^{-2i\theta(\gamma)(n-\frac{m}{2}+\frac{1}{2})}}{n-\frac{m}{2}+s_1}-\frac{e^{2i\theta(\gamma)(n+\frac{m}{2}+\frac{1}{2})}}{n+\frac{m}{2}+s_1}\Bigg)\\
	-\frac{i}{s_2-\frac{1}{2}}\sum_{n=0}^\infty\Bigg(\frac{e^{-2i\theta(\gamma)(n-\frac{m}{2}+\frac{1}{2})}}{n-\frac{m}{2}+s_2}-\frac{e^{2i\theta(\gamma)(n+\frac{m}{2}+\frac{1}{2})}}{n+\frac{m}{2}+s_2}\Bigg)\Bigg].\label{eq:EllipticFinal}
\end{multline}

\begin{prop}\label{prop:LMero}
	The logarithmic derivative $L(s;\rho)$ of the twisted Selberg zeta function $Z(s;\rho)$ has a meromorphic extension to the whole complex plane~$\C$. Its poles are simple and given by the following formal expression:
	\begin{multline*}
		\sum_{j=0}^\infty\left[\frac{1}{s-\frac{1}{2}-i\mu_j}+\frac{1}{s-\frac{1}{2}+i\mu_j}\right]\\
		+\frac{\Vol(X)\dim(V_\rho)}{4\pi}\sum_{n=0}^\infty\left[\frac{1-2s}{s+\frac{m}{2}+n}+\frac{1-2s}{s-\frac{m}{2}+n}\right]\\
		+i\sum_{[\gamma]\textup{ ell.}}\frac{\tr\rho(\gamma)}{2M(\gamma)\sin(\theta(\gamma))}\sum_{n=0}^\infty\left[\frac{e^{2i\theta(\gamma)(n+\frac{m}{2}+\frac{1}{2})}}{s+\frac{m}{2}+n}-\frac{e^{-2i\theta(\gamma)(n-\frac{m}{2}+\frac{1}{2})}}{s-\frac{m}{2}+n}\right]
	\end{multline*}
	where $(\frac{1}{4}+\mu_j^2)_{j\in\Z_{\geq0}}\subseteq\C$ are the eigenvalues of $\Delta_{\tau_m,\rho}^\sharp$, counted with algebraic multiplicity.
\end{prop}

\begin{proof}
	Fix $s_2\in\C$ with $\RE\big((s_2-\frac{1}{2})^2\big)$ sufficiently large and let $s=s_1\in\C$. Multiplying \cref{eq:ResolventIdentityWithIdHypEllContributions} by $2(s_1-\frac{1}{2})$ gives
	\begin{multline*}
		L(s;\rho) = 2(s-\tfrac{1}{2})\tr\left(\left(A_\rho^\sharp+(s-\tfrac{1}{2})^2\right)^{-1}-\left(A_\rho^\sharp+(s_2-\tfrac{1}{2})^2\right)^{-1}\right)\\
		-2(s-\tfrac{1}{2})I(s,s_2;\rho)-2(s-\tfrac{1}{2})E(s,s_2;\rho)+\frac{s-\frac{1}{2}}{s_2-\frac{1}{2}}L(s_2;\rho).
	\end{multline*}
	Together with \cref{eq:IdentityFinal} and \cref{eq:EllipticFinal}, this shows that $L(s;\rho)$ extends meromorphically to all $s\in\C$. The contribution of the first two terms to the poles of $L(s;\rho)$ was obtained in \cite[Proposition 4.2.5]{FS}. The last term is holomorphic in $s$, and for the term $E(s,s_2;\rho)$ one can use \cref{eq:EllipticFinal}.
\end{proof}

\begin{thm}[Meromorphic continuation of the Selberg zeta function]
\label{theo:mero}
	The twisted Selberg zeta function $Z(s;\rho)$ has a meromorphic extension to the whole complex plane $\C$ with poles and zeros given by the formal product
	\begin{multline*}
		\prod_{j=0}^\infty\left(s-\frac{1}{2}-i\mu_j\right)\left(s-\frac{1}{2}+i\mu_j\right)\\
		\times\prod_{n=0}^\infty\left(s+\frac{m}{2}+n\right)^{N_+(m,n)}\left(s-\frac{m}{2}+n\right)^{N_-(m,n)},
	\end{multline*}
	where $(\frac{1}{4}+\mu_j^2)_{j\in\Z_{\geq0}}\subseteq\C$ are the eigenvalues of $\Delta_{\tau_m,\rho}^\sharp$, counted with algebraic multiplicity, and
	\begin{multline*}
		N_\pm(m,n) = \frac{\Vol(X)\dim(V_\rho)}{4\pi}(\pm m+2n+1)\\
		\pm i\sum_{[\gamma]\textup{ ell.}}\frac{\tr\rho(\gamma)}{2M(\gamma)\sin(\theta(\gamma))}e^{\pm2i\theta(\gamma)(n\pm\frac{m}{2}+\frac{1}{2})}.
	\end{multline*}
\end{thm}

\begin{proof}
	In view of \cref{prop:LMero} it suffices to show that the residues of $L(s;\rho)$ are integers. For the poles at $s=\frac{1}{2}\pm i\mu_j$ the residue is the algebraic multiplicity of $\mu_j$. For $s=\mp\frac{m}{2}-n$ the residue is $N_\pm(m,n)$ which is an integer by \cref{thm:DimensionEigenspaces}. If these two poles coincide, one has to add up the two residues.
\end{proof}

\begin{rmrk}
	Since $\Delta_{\tau_m,\rho}^\sharp$ is in general not self-adjoint, its eigenvalues might be complex, although contained in a positive cone in $\C$. It may therefore happen that finitely many of the singularities $\frac{1}{2}\pm i\mu_j$ of $L(s;\rho)$ coincide with finitely many of the singularities $\pm\frac{m}{2}-n$. In particular, we are not able to decide more precisely what the poles and zeros of the Selberg zeta function in \cref{theo:mero} are.
\end{rmrk}

\begin{coro}\label{coro:SelbergVanishingAtZero}
	The order of vanishing of the Selberg zeta function $Z(s;\rho)$ at $s=0$ is
	$$ \mult(0;\Delta_{\tau_m,\rho}^\sharp) + \begin{cases}(2g-2+r)\dim(V_\rho)&\mbox{for $m=0$,}\\0&\mbox{for $m\neq0$.}\end{cases} $$
\end{coro}

\begin{proof}
	By \cref{theo:mero} the order of vanishing equals the multiplicity of the eigenvalue $0$ of $\Delta_{\tau_m,\rho}^\sharp$ plus additionally $N_+(0,0)+N_-(0,0)$ in the case $m=0$. Note that
	$$ N_+(0,0)+N_-(0,0) = \frac{\Vol(X)\dim(V_\rho)}{2\pi}-\sum_{[\gamma]\textup{ ell.}}\frac{\tr\rho(\gamma)}{M(\gamma)}. $$
	To compute the sum, we group the elliptic elements into the subsets $\{c_j,c_j^2,\ldots,c_j^{\nu_j-1}\}$ for $j=1,\ldots,r$. Fix $j=1,\ldots,r$ and let $\alpha$ be an eigenvalue of $\rho(c_j)$, then the contribution of $\alpha$ to the sum $\sum\tr\rho(\gamma)$ is
	$$ \sum_{i=1}^{\nu_j-1}\alpha^i = \frac{\alpha-\alpha^{\nu_j}}{1-\alpha} = -1, $$
	because $\rho(c_j)^{\nu_j}=\rho(c_j^{\nu_j})=\rho(u)=\Id$. It follows that
	$$ \sum_{i=1}^{\nu_j-1}\frac{\tr(c_j^i)}{M(c_j^i)} = -\frac{\dim(V_\rho)}{\nu_j} $$
	and hence
	$$ N_+(0,0)+N_-(0,0) = \frac{\Vol(X)\dim(V_\rho)}{2\pi}+\dim(V_\rho)\sum_{j=1}^r\frac{1}{\nu_j}. $$
	The claim now follows from the Gauss--Bonnet theorem for $X$ which asserts that
	\begin{equation}
		\frac{\Vol(X)}{2\pi} = -\chi(X) = 2g-2+r-\sum_{j=1}^r\frac{1}{\nu_j}.\label{eq:GaussBonnet}\qedhere
	\end{equation}
\end{proof}

\subsection{The functional equation}
\label{subsec:FE}
From the trace formula we further conclude:

\begin{thm}[Functional equation]
\label{theo:Funct}
	The twisted Selberg zeta function $Z(s;\rho)$ satisfies the following functional equation:
	\begin{multline}
		\eta(s;\rho) = \frac{Z(s;\rho)}{Z(1-s;\rho)} = \exp\Bigg(\int_0^{s-\frac{1}{2}}\dim(V_\rho)\Vol(X)\frac{\xi\sin(2\pi\xi)}{\cos(2\pi\xi)+\cos(\pi m)}\\
		- \sum_{[\gamma]\textup{ ell.}}\frac{\pi\tr\rho(\gamma)}{M(\gamma)\sin(\theta(\gamma))}\frac{\cos(2(\pi-\theta(\gamma))\xi)+e^{i\pi m}\cos(2\theta(\gamma)\xi)}{\cos(2\pi\xi)+\cos(\pi m)}\,d\xi\Bigg),\label{eq:FunctionalEquationSelberg}
	\end{multline}
	where the integral is along any contour from $0$ to $s-\frac{1}{2}$ in the complex plane avoiding the zeros of the denominator.
\end{thm}

\begin{rmrk}
Note that the residues of the integrand in \cref{eq:FunctionalEquationSelberg} are integers and therefore it does not matter which contour from $0$ to $s-\frac{1}{2}$ is used for the integral. In fact, the integrand has poles at $\xi\in\frac{1}{2}\pm\frac{m}{2}+\Z$, and using the identity
$$ \Res\left(\frac{f'(\xi)}{f(\xi)},\xi_0\right) = \operatorname{ord}_{\xi_0}(f) $$
for $f(\xi)=\cos(2\pi\xi)+\cos(\pi m)$, it can be shown that the residue at $\xi=\frac{1}{2}\pm\frac{m}{2}+n$ equals $\operatorname{ord}_{\xi_0}(f)$ times
$$ -\frac{\Vol(X)\dim(V_\rho)}{4\pi}(\pm m+2n+1) \mp i\sum_{[\gamma]\\\textup{ ell.}}\frac{\tr\rho(\gamma)}{2M(\gamma)\sin(\theta(\gamma))}e^{\pm2i\theta(\gamma)(n\pm\frac{m}{2}+\frac{1}{2})} $$
which is an integer by \cref{thm:DimensionEigenspaces}.
\end{rmrk}

\begin{proof}
	After having computed the right hand side of \cref{eq:ResolventIdentityWithIdHypEllContributions} explicity, it is obviously meromorphic in $s_1,s_2\in\C$. We can therefore plug in $s_1=s$ and $s_2=1-s$ to obtain
	$$ 0 = I(s,1-s;\rho)+E(s,1-s;\rho)+H(s,1-s;\rho). $$
	The identity contribution equals
	\begin{align*}
		I(s,1-s;\rho) &= -\frac{\dim(V_\rho)\Vol(X)}{4}\big(\tan(\pi(s-\tfrac{1}{2}+\tfrac{m}{2}))+\tan(\pi(s-\tfrac{1}{2}-\tfrac{m}{2}))\big),
	\end{align*}
	by \cite[proof of Theorem 4.2.8]{FS}. Using
	\begin{equation*}
		\tan(x+y)+\tan(x-y) = \frac{2\sin(2x)}{\cos(2x)+\cos(2y)},\label{eq:TangentIdentity}
	\end{equation*}
	this becomes
	$$ I(s,1-s;\rho) = -\frac{\dim(V_\rho)\Vol(X)}{2}\frac{\sin(2\pi(s-\frac{1}{2}))}{\cos(2\pi(s-\frac{1}{2}))+\cos(\pi m)}. $$
	The hyperbolic contribution becomes
	$$ H(s,1-s;\rho) = \frac{L(s;\rho)+L(1-s;\rho)}{2(s-\frac{1}{2})} $$
	and the elliptic contribution is
	\begin{align*}
		& E(s,1-s;\rho)\\
		& =\sum_{[\gamma]\textup{ ell.}} \frac{\tr\rho(\gamma)}{4M(\gamma)\sin(\theta(\gamma))}\Bigg[\frac{i}{s-\frac{1}{2}}\sum_{n\in\Z}\Bigg(\frac{e^{2i\theta(\gamma)(n+\frac{m}{2}+\frac{1}{2})}}{s-\frac{m}{2}-n-1}-\frac{e^{2i\theta(\gamma)(n+\frac{m}{2}+\frac{1}{2})}}{s+\frac{m}{2}+n}\Bigg)\Bigg]\\
		& =\sum_{[\gamma]\textup{ ell.}} \frac{\pi\tr\rho(\gamma)}{2M(\gamma)(s-\frac{1}{2})\sin(\theta(\gamma))}\frac{\cos(2(\pi-\theta(\gamma))(s-\frac{1}{2}))+e^{i\pi m}\cos(2\theta(\gamma)(s-\frac{1}{2}))}{\cos(2\pi(s-\frac{1}{2}))+\cos(\pi m)}
	\end{align*}
	by \cite[equation (**) on page 444]{He2}.
	This implies
	\begin{multline} 
	\label{eq:logeta}
		\frac{d}{ds}\log\frac{Z(s;\rho)}{Z(1-s;\rho)} = L(s;\rho)+L(1-s;\rho)\\
		= \dim(V_\rho)\Vol(X)\frac{(s-\frac{1}{2})\sin(2\pi(s-\frac{1}{2}))}{\cos(2\pi(s-\frac{1}{2}))+\cos(\pi m)}\\
		- \sum_{[\gamma]\textup{ ell.}}\frac{\pi\tr\rho(\gamma)}{M(\gamma)\sin(\theta(\gamma))}\frac{\cos(2(\pi-\theta(\gamma))(s-\frac{1}{2}))+e^{i\pi m}\cos(2\theta(\gamma)(s-\frac{1}{2}))}{\cos(2\pi(s-\frac{1}{2}))+\cos(\pi m)}
	\end{multline}
	and the claim follows by integration and exponentiation.
\end{proof}

\subsection{Fried's conjecture}
\label{subsec:Fried}
We compute the behavior of the Ruelle zeta function at $s=0$, and we prove \cref{theo:mainFried}.

\medbreak

We start with the following proposition:
\begin{prop}\label{prop:RuelleVsEtaAtZero}
The Ruelle zeta function extends meromorphically to the whole complex plane. Moreover, 
$$R(s; \rho) \sim_{s \to 0} \eta(s+1;\rho)^{-1}\times\begin{cases}(-1)^{\mult(0;\Delta_{\tau_m,\rho}^\sharp)+(2g-2+r)\dim(V_\rho)}&\mbox{for $m=0$,}\\(-1)^{\mult(0;\Delta_{\tau_m,\rho}^\sharp)}&\mbox{for $m\neq0$,}\end{cases} $$
with $\eta(s;\rho)$ as in \cref{theo:Funct}.
\end{prop}

We remark that a proof of the fact that $R(s;\rho)$ extends meromorphically is already sketched in \cite[Theorem 1]{fried1986fuchsian}. However, Fried's proof is quite general and does not provide any information about the location and nature of poles and zeros of $R(s;\rho)$.

\begin{proof}
Using \cref{eq:RuelSel} and \cref{theo:mero} we directly deduce that $R(s; \rho)$ extends meromorphically. Together with the functional equation for the Selberg zeta function established in \cref{theo:Funct}, we get
$$R(s; \rho) = \frac{Z(s;\rho)}{Z(s+1;\rho)} = \frac{Z(s;\rho)}{\eta(s+1;\rho) Z(-s;\rho)}$$
and the result follows from \cref{coro:SelbergVanishingAtZero}.
\end{proof}

Now, we are led to compute the value of $\eta(s+1;\rho)$ at $s=0$. The rest of this section is devoted to this computation. We summarize the result in the following theorem. Recall that each singular point in the orbisurface $X$ is represented by a loop $c_j \in \pi_1(X_1)$ such that $c_j^{\nu_j} = u$. 
For any $j=1, \ldots, r$ we denote by $n_j= \dim \Fix  \rho(c_j)$, so that $\rho(c_j) = I_{n_j} \oplus T_j$.

\begin{thm}
\label{theo:Ruelle0}
For $\rho\colon \pi_1(X_1) \to \GL(V_\rho)$ an irreducible representation:
\begin{enumerate}
\item \label{item1}
If $\rho(u)= Id_{V_\rho}$, then
$$ R\left(\frac{s}{2\pi}; \rho\right) \sim_{s \to 0}  (-1)^{\mult(0;\Delta_{\tau_m,\rho}^\sharp)}\frac{s^{n(2g-2+r)-\sum_{j=1}^rn_j}}{\prod_{j=1}^r |\det(I_{n-n_j} -T_j)|\,(-\nu_j)^{-n_j}}. $$
\item \label{item2}
Otherwise
$$R(0;\rho) = (-1)^{\mult(0;\Delta_{\tau_m,\rho}^\sharp)}\frac{\det(\rho(u)-I_n)^{2g+r-2}}{\prod_{j=1}^r \det(\rho(c_j)-I_n)}.$$
\end{enumerate}
\end{thm}

\begin{rmrk}\label{rmk:SignInM=0Formula}
	It is possible to remove in \cref{theo:Ruelle0} \cref{item1} the absolute value from $\det(I_{n-n_j}-T_j)$ at the cost of a sign which we are only able to express in terms of the eigenvalues of the operators $\rho(c_j)$. More precisely, we claim that
	\begin{equation}
		\prod_{j=1}^r|\det(I_{n-n_j}-T_j)| = \pm\prod_{j=1}^r i^{n-n_j}\det(I_{n-n_j}-T_j),\label{eq:AbsValueDetAsPlusMinusDet}
	\end{equation}
	where the sign can be determined as follows: For every $j$ we let $\det(T_j)^{1/2}=\det(\rho(c_j))^{1/2}$ denote the square root of $\det(T_j)=\det(\rho(c_j))$ which is the product of $e^{i\theta/2}$, where $e^{i\theta}$, $\theta\in(0,2\pi)$, runs through the eigenvalues of $T_j$. Since $\prod_{j=1}^r\rho(c_j)=\rho(u)=I_n$, we have $\prod_{j=1}^r\det(T_j)=1$ and therefore $\prod_{j=1}^r\det(T_j)^{1/2}=\pm1$. This is the sign in \cref{eq:AbsValueDetAsPlusMinusDet}.
	
	To prove this claim, we let $\{e^{i\theta_\alpha}\}_\alpha$, $\theta_\alpha\in(0,2\pi)$, denote the set of eigenvalues of the operators $T_j$ with multiplicities. Then
	\begin{multline*}
		\prod_{j=1}^r\det(I_{n-n_j}-T_j) = \prod_\alpha(1-e^{i\theta_\alpha}) = \exp\left(i\sum_\alpha\arg(1-e^{i\theta_\alpha})\right)\prod_\alpha|1-e^{i\theta_\alpha}|\\
		= \exp\left(i\sum_\alpha\arg(1-e^{i\theta_\alpha})\right)\prod_{j=1}^r|\det(I_{n-n_j}-T_j)|.
	\end{multline*}
	Using the formula
	\begin{equation}
		\arg(1-e^{i\theta}) = \frac{\theta-\pi}{2} \qquad \mbox{for all }\theta\in(0,2\pi),\label{eq:FormulaForArgument1MinusZ}
	\end{equation}
	we find
	$$ \sum_\alpha\arg(1-e^{i\theta_\alpha}) = \sum_\alpha\frac{\theta_\alpha-\pi}{2} = \sum_\alpha\frac{\theta_\alpha}{2}-\frac{\pi}{2}\sum_{j=1}^r(n-n_j) $$
	and hence
	\begin{align*}
		\exp\left(i\sum_\alpha\arg(1-e^{i\theta_\alpha})\right) &= i^{-\sum_{j=1}^r(n-n_j)}\prod_\alpha e^{i\theta_\alpha/2}\\
		&= i^{-\sum_{j=1}^r(n-n_j)}\prod_{j=1}^r \det(T_j)^{1/2}
	\end{align*}
	as claimed.
\end{rmrk}

Now we prove \cref{theo:Ruelle0} by computing the behavior of $\eta(1+s;\rho)$ near $s=0$. We start with the following observation:
\begin{lem}
We have
\begin{equation}
\label{lem:eta}
\eta(1+s;\rho) = \exp \int_0^{s+\frac 1 2} \frac {\eta'(\xi+\frac 1 2;\rho)}{\eta(\xi+ \frac 1 2;\rho)} d\xi.
\end{equation}
\end{lem}
\begin{rmrk}
\label{remk:residues}
Note that the previous expression is independent of the chosen path between $0$ and $s+ \frac 1 2$, since the residues of the logarithmic derivative of the meromorphic function $\eta$ are integers.
\end{rmrk}

\begin{proof}
Note that the equality holds true at $s=-\frac{1}{2}$, where indeed both side are equal to 1. Let $s\in \C$ such that $\eta(1+s;\rho)$ is neither $0$ nor $\infty$. Then, taking a determination of the logarithm and differentiating both sides, we find that both sides of the equation locally coincide, up to a multiplicative constant. Since both sides define a meromorphic function, the constant should be global. 
Considering $s=-\frac{1}{2}$, we conclude that it is equal to 1.
\end{proof}

Let $m \in [-1,1]$ such that $\rho(u) = e^{-i \pi m}$, so that $m=0$ corresponds to Case \ref{item1} of \cref{theo:Ruelle0}, and $m\neq 0$ corresponds to Case \ref{item2}.
In order to compute the integral in \cref{lem:eta}, we group the terms as in \cref{eq:logeta}:
\begin{equation}
\label{eq:sum}
\frac {\eta'(\xi+\frac 1 2;\rho)}{\eta(\xi+ \frac 1 2;\rho)} = \frac{-\dim(V_\rho) \Vol(X)}{2\pi} \frac {-2\pi \xi \sin(2\pi \xi)}{\cos(2\pi\xi) + \cos(\pi m)} + \sum_j E_j(\xi),
\end{equation}
where 
\begin{equation}
\label{eq:Ei}
E_j(\xi) = -\sum_{k=1}^{\nu_j -1}\frac{\pi  \tr \rho(c_j^k)}{\nu_j \sin (k \pi/\nu_j)} \frac{\cos \left(2(\pi-k\pi/\nu_j)\xi\right) + e^{i\pi m}\cos (2k\pi \xi/\nu_j)}{\cos(2\pi \xi)+ \cos(\pi m)}
\end{equation}
groups all the summands corresponding to the elliptic element $c_j$ and its powers. 

In view of \cref{remk:residues}, one needs to make a choice of an integration path avoiding the poles of the integrand in \cref{lem:eta}. The integrand has a unique pole in $[0, \frac 1 2]$, and it is given by $\xi_0=\frac{1-|m|}{2}$. We choose a path from $0$ to $\frac{1}{2}$ which avoids $\xi_0$ along a small half-circle in the lower-half plane and we denote this path by $C_-$, oriented following the increasing real part. We take $C_+ = \overline {C_-}$ the mirror image of the path $C_-$ after a reflection through the real line, then $C = C_- - C_+$ is a small circle encircling $\xi_0$ counterclockwise.

First, we compute the contribution of the identity, namely the integral of the first summand in the right-hand side of \cref{eq:sum}.
\begin{lem}
We have
\label{lem1}
\begin{align}
\label{eq:ident}
\exp \left(\int_{C_-} \frac{-2\pi\xi \sin(2\pi \xi)\,d\xi}{\cos(2\pi \xi) + \cos(\pi m)}\right) &= 1-e^{-i\pi|m|} &&\text{for } m \neq 0,\\
\exp\left( \int_0^{1/2+s} \frac{-2\pi\xi \sin(2\pi \xi)\,d\xi}{\cos(2\pi \xi) +1}\right) &\sim_{s \to 0} -2\pi s + o(s) && \text{for } m=0.\nonumber
\end{align}
\end{lem}
\begin{proof}
The asymptotic case (for $m=0$) was already proven in \cite[Lemma 3]{fried1986fuchsian}, so we just show \cref{eq:ident} (for $m\neq 0$). Assume $m\neq\pm1$, the case $m=\pm1$ only needs minor modification.
By \cite[Lemma 2]{fried1986fuchsian}, we know that 
$$\RE\left(\int_{C_-} \frac{-2\pi\xi \sin(2\pi \xi)\,d\xi}{\cos(2\pi \xi) + \cos(\pi m)}\right) = \log |1-e^{i\pi m}|=\log |1-e^{-i\pi m}|.$$
Hence, we only need to show that
$$\Im \left(\int_{C_-} \frac{-2\pi\xi \sin(2\pi \xi)\,d\xi}{\cos(2\pi \xi)+ \cos(\pi m)}\right) = 
\arg(1-e^{-i\pi|m|}).$$

Let $f(\xi) =  \cos(2\pi \xi) + \cos(\pi m)$. We want to compute the imaginary part of the integral 
$$\int_{C_-} \xi \frac{f'(\xi)}{f(\xi)} d\xi.$$
Note that the function $\xi \mapsto \xi \frac{f'(\xi)}{f(\xi)}$ sends the real axis to itself, hence we have 
$\int_{C_+}  \xi \frac{f'(\xi)}{f(\xi)} d\xi = \overline{\int_{C_-}  \xi \frac{f'(\xi)}{f(\xi)} d\xi}$ and 
$\Im \left(\int_{C_+} \xi \frac{f'(\xi)}{f(\xi)} d\xi\right) = -\Im \left(\int_{C_-} \xi \frac{f'(\xi)}{f(\xi)} d\xi\right)$. We deduce that 
$$\Im \left(\int_{C_-} \xi \frac{f'(\xi)}{f(\xi)} d\xi\right) = \frac 1 2 \Im \left(\int_{C} \xi \frac{f'(\xi)}{f(\xi)} d\xi\right) = \pi \Res\left(\xi \frac{f'(\xi)}{f(\xi)}, \xi_0\right)= \pi \xi_0.$$
A simple trigonometric computation shows
$\arg(1-e^{-i\pi|m|}) = \pi\xi_0$, hence \cref{eq:ident} follows.
\end{proof}

Now, we compute the terms coming from the elliptic elements. For this we first introduce some notation. Let
$$ \log:\C\setminus(-\infty,0]\to\{z\in\C:|\IM z|<\pi\} $$
denote the principal branch of the logarithm. For any $\theta\in\R$ we introduce the branch
$$ \log_\theta:\C\setminus\{-re^{i\theta}:r\geq0\}\to\{z\in\C:|\IM z-\theta|<\pi\}, \,\, \log_\theta(w) = \log(e^{-i\theta}w)+i\theta. $$
Writing $\log_\theta(z)=\log|z|+i\arg_\theta(z)$ and $\log(z)=\log|z|+i\arg(z)$, the functional equation for the logarithm holds under a certain assumption on the arguments:
\begin{equation}\label{eq:FunctionalEqLog}
	\log_\theta(z)-\log_\theta(w) = \log\left(\frac{z}{w}\right) \qquad \mbox{whenever }\arg_\theta(z)-\arg_\theta(w)\in(-\pi,\pi).
\end{equation}

We first treat the case $m\neq0$, which generalizes \cite[Lemma 5]{fried1986fuchsian}:

\begin{lem}
\label{lem2}
For $m\neq0$ we have
\begin{align}
\label{eq:elliptic}
\exp \int_{C_-} E_j  &= \frac {\det(I_n-\rho(c_j)^{\sign(m)})}{\det(I_n-\rho(u)^{\sign(m)})^{1/\nu_j} }.
\end{align}
\end{lem}
\begin{proof}
We follow the lines of \cite[Proof of Lemma 5]{fried1986fuchsian}, keeping track of the arguments of the terms involved.
For each $j$, recall that the contribution $E_j$ of the elliptic element $c_j$ and its powers is given in \cref{eq:Ei}. The term $\Tr \rho(c_j)^k$ is $\sum_\alpha  \alpha^k$, for $\alpha$ running through the eigenvalues of $\rho(c_j)$. In particular $\alpha^M = e^{-i\pi m}$, where $M = \nu_j$ is the order of $c_j$ (recall that $\rho(u) = e^{-i\pi m}$). We compute the contribution of each $\alpha$ to $\int_{C_-}E_j$ as follows:
we set $z_0 = e^{i\pi/M}$ and $z = e^{2i\pi \xi/M}$ and get that $\alpha$ contributes 
$\int_{\mathcal C_-} \frac{P(z)}{Q(z)} dz$, where
\begin{align*}
	P(z) &= - \frac 1 {2iz} \sum_{k=1}^{M-1} \frac {\alpha^k}{\sin(k\pi/M)} (z^{M-k}+z^{k-M} + e^{i\pi m} (z^k+z^{-k})) z^M,\\
	Q(z) &= (z^M + e^{i\pi m})(z^M + e^{-i\pi m})
\end{align*}
are polynomials in $z$, and $\mathcal C_-$ is the image of the path $C_-$, by the map $\xi \mapsto e^{2i\pi \xi/M}$. It is a small deformation of an arc of the unit circle, avoiding the pole occurring at $e^{2i\pi \xi_0/M}$ by moving around it outside of the unit circle.

Indeed, one can rearrange $P$ by grouping the terms corresponding to $k$ and $M-k$ in the sum, factorizing $(z^k + z^{-k})$ and re-expanding:
$$P(z) = - \frac 1 {2iz} \sum_{k=1}^{M-1} \frac {\alpha^k e^{i\pi m} + \alpha^{-k}e^{-i\pi m}}{\sin(k\pi/M)} (z^k+z^{-k}) z^M.$$
Let us assume that $m \neq \pm 1$, so that the polynomial $Q$ splits as the product $Q(z) = \prod_{l=1}^{2M} (z-z_l)$ with simple roots $z_1,\ldots,z_{2M}$. Hence, we can decompose the rational fraction $P/Q$ as
$$\frac{P(z)}{Q(z)} = \sum_{l=1}^{2M} \frac{P(z_l)/Q'(z_l)}{z-z_l}.$$
To integrate $(z-z_l)^{-1}$, we choose the branch $\log_\theta$ of the logarithm with $\theta=\frac{2\pi\xi_0}{M}$. It is easy to see that $z-z_l$ is in the domain of $\log_\theta$ for all $z\in\mathcal C_-$ and all $l$. Since $m \notin \Z$, one can integrate to obtain 
\begin{equation}
	\int_{\mathcal C_-} \frac{P(z)}{Q(z)}dz= \sum_{l=1}^{2M} \frac {P(z_l)}{Q'(z_l)} \big(\log_\theta(z_0-z_l)-\log_\theta(1-z_l)\big).\label{eq:IntPoverQ}
\end{equation}

Now
$$\frac{P(z_l)}{Q'(z_l)} = \sum_{k=1}^{M-1}  \frac {(\alpha^k e^{i\pi m} + \alpha^{-k}e^{-i\pi m})(z_l^k+z_l^{-k})}{2iM \varepsilon (e^{-i\pi m}-e^{i\pi m}) \sin(k\pi/M)},$$
where $\varepsilon = \pm 1$ depending on $z_l^M + e^{-\varepsilon i\pi m}=0$. In particular, this term is real, and the term for $z_l^{-1}$ is equal to the negative of the term for $z_l$. Hence, the contribution of $z_l$ and $z_l^{-1}$ together to \cref{eq:IntPoverQ} equals $P(z_l)/Q'(z_l)$ times
\begin{equation}
	\log_\theta(z_0-z_l)-\log_\theta(1-z_l)-\log_\theta(z_0-z_l^{-1})+\log_\theta(1-z_l^{-1}).\label{eq:LogTermsGrouped}
\end{equation}
First note that by \cref{eq:FunctionalEqLog}, we have
$$ \log_\theta(1-z_l)-\log_\theta(1-z_l^{-1}) = \log\left(\frac{1-z_l}{1-z_l^{-1}}\right) = \log(-z_l). $$
Moreover, again by \cref{eq:FunctionalEqLog}:
$$ \log_\theta(z_0-z_l) = \log_\theta(z_0) + \log(1-\frac{z_l}{z_0}) \quad \mbox{and} \quad \log_\theta(z_0-z_l^{-1}) = \log_\theta(z_0) + \log(1-\frac{1}{z_0z_l}), $$
so that \cref{eq:LogTermsGrouped} becomes
$$ \log_\theta(1-\frac{z_l}{z_0})-\log_\theta(1-\frac{1}{z_0z_l})-\log(-z_l). $$
Finally, using \cref{eq:FunctionalEqLog} once more:
$$ \log_\theta(1-\frac{1}{z_0z_l}) = \log_\theta(1-z_0z_l) + \log(\frac{1-\frac{1}{z_0z_l}}{1-z_0z_l}) = \log_\theta(1-z_0z_l) + \log(-z_0^{-1}z_l^{-1}). $$
Since both $1-\frac{z_l}{z_0}$ and $1-z_0z_l$ are contained in the right half plane and the branch cut of $\log_\theta$ is in the left half plane, $\log_\theta=\log$ in this case. Summing up, we obtain the following expression for \cref{eq:LogTermsGrouped}:
$$ \log(1-\frac{z_l}{z_0})-\log(1-z_0z_l)-\log(-z_l)-\log(-z_0^{-1}z_l^{-1}). $$
It remains to determine
$$ \log(-z_l)+\log(-z_0^{-1}z_l^{-1}) = \log(-z_l)-\log(-z_0z_l) $$
for each $l$. For this purpose we choose $\varepsilon=\sign(m)$, so $z_l^M+e^{-i\pi|m|}=0$. This implies that $\arg(-z_l)\not\in(\pi-\frac{\pi}{M},\pi)$, so $-z_l$ and $-z_0z_l$ lie on the same side of the branch locus of $\log$ and hence
$$ \log(-z_l)-\log(-z_0z_l) = -\log(z_0) = -\frac{i\pi}{M}. $$
In total, the contribution of $z_l$ and $z_l^{-1}$ to \cref{eq:IntPoverQ} equals $P(z_l)/Q'(z_l)$ times
\begin{equation}
 \frac {i\pi}{M} + \log(1-\frac {z_l}{z_0}) - \log(1-z_l z_0).\label{eq:LogRearrangement}
\end{equation}

We first note that the contribution of $\frac{i\pi}{M}$ is trivial. In fact, its contribution to $\int_{\mathcal C_+} \frac{P(z)}{Q(z)}dz$ equals
$$ \frac{i\pi}{M}\sum_{z_l^M=-e^{-i\pi m}}\sum_{k=1}^{M-1}\frac {(\alpha^k e^{i\pi m} + \alpha^{-k}e^{-i\pi m})(z_l^k+z_l^{-k})}{2iM\varepsilon(e^{-i\pi m}-e^{i\pi m}) \sin(k\pi/M)}, $$
and interchanging the sums and using that
$$ \sum_{z_l^M=-e^{-i\pi m}}z_l^k = \sum_{z_l^M=-e^{-i\pi m}}z_l^{-k} = 0 $$
shows that this term indeed vanishes. To express the contribution of the second and third term in \cref{eq:LogRearrangement}, we set $\omega = \frac {z_l}{z_0}$ in the second term and $\omega = z_l z_0$ in the third term. Then, summation is over $\omega^M=e^{-i\pi|m|}$ and we obtain:
\begin{align*}
\int_{\mathcal C_-} \frac{P(z)}{Q(z)} dz ={}& \sum_{\omega^M=e^{-i\pi |m|}} \log(1-\omega) \sum_{k=1}^{M-1}  \frac {(\alpha^k e^{i\pi m} + \alpha^{-k}e^{-i\pi m})}{2iM\varepsilon(e^{-i\pi m} - e^{i\pi m}) \sin(k\pi/M)}\\
&\hspace{3.8cm}\times(z_0^k\omega^k + z_0^{-k} \omega^{-k} - z_0^{-k} \omega^k - z_0^k \omega^{-k})\\
={}& \sum_{\omega^M=e^{-i\pi |m|}} \log(1-\omega) \sum_{k=1}^{M-1}  \frac {(\alpha^k e^{i\pi m} + \alpha^{-k}e^{-i\pi m})}{M\varepsilon(e^{-i\pi m} - e^{i\pi m})}(\omega^k-\omega^{-k}).
\end{align*}
Note that $\omega=\zeta \alpha^\varepsilon$, with $\zeta^M=1$. In the sum over $k$ above, the term for $k$ simplifies with the term for $M-k$ and the remaining terms equal
\begin{multline*}
	\frac 1 {M\varepsilon(e^{-i\pi m} - e^{i\pi m})} \sum_{k=1}^{M-1}\big(\alpha^{-\varepsilon k}\omega^ke^{-i\pi|m|}-\alpha^{\varepsilon k}\omega^{-k}e^{i\pi|m|}\big)\\
	= \frac 1 {M(e^{-i\pi m} - e^{i\pi m})} \sum_{k=1}^{M-1} \big(e^{-i\pi m}\zeta^k - e^{i\pi m} \zeta^{-k}\big) = \begin{cases} -1/M & \text{if } \zeta \neq 1,\\ (M-1)/M & \text{if } \zeta = 1.\end{cases}.
\end{multline*}
Hence, we obtain 
$$\int_{\mathcal C_-} \frac{P(z)}{Q(z)}dz = \log(1-\alpha^\varepsilon) - \frac 1 M \sum_{\omega^M = e^{-i\pi |m|}} \log (1-\omega).$$
We claim that the last sum above is just $\log (1-e^{-i\pi |m|})$:

\begin{claim}
	For any $z\in\C$, $|z|=1$, $z\neq1$, and $M\geq1$ the following identity holds:
	\begin{equation}
		\sum_{\omega^M = z} \log (1-\omega) = \log (1-z).
	\end{equation}
	\end{claim}
	\begin{proof}[Proof of the claim]
		Since $\log(z)+\log(w)-\log(zw)\in2\pi i\Z$ for all $z,w\in\C$ with $z,w,zw\not\in(-\infty,0]$, we find
		$$ \sum_{\omega^M = z} \log (1-\omega) \equiv \log\Big(\prod_{\omega^M = z} (1-\omega)\Big) = \log (1-z) \qquad \mod 2\pi i\Z. $$
		It remains to show that the imaginary parts of both sides are actually equal. Since the argument of $1-z$ is contained in $(-\frac{\pi}{2},\frac{\pi}{2})$, it suffices to show that
		$$ \sum_{\omega^M = z}\arg (1-\omega) \in \left(-\frac{\pi}{2},\frac{\pi}{2}\right). $$
		The $M$-th roots of $z$ can be written as $e^{i\theta}$ with $\theta=\phi+\frac{2\pi k}{M}$ for $k=0,\ldots,M-1$ and some $\phi\in(0,\frac{2\pi}{M})$. Hence, by \cref{eq:FormulaForArgument1MinusZ}:
		\begin{equation*}
			\sum_{\omega^M = z}\arg (1-\omega) = \sum_{k=0}^{M-1} \frac{\phi+\frac{2\pi k}{M}-\pi}{2} = \frac{M\phi}{2}-\frac{\pi}{2} \in \left(-\frac{\pi}{2},\frac{\pi}{2}\right).\qedhere
		\end{equation*}
	\end{proof}

Adding all this for every $\alpha$ and taking the exponential, we obtain \cref{eq:elliptic}.

For $m=\pm 1$, one needs to be a bit more careful when decomposing $Q(z)$. However, the computation goes through, combining what we have just done and \cite[Section 3]{fried1986fuchsian}.
Hence, we omit this case.
\end{proof}

Now we generalize \cite[Lemma 6]{fried1986fuchsian}, which computes the elliptic contribution in the case $m=0$:

\begin{lem}
	\label{lem3}
	For $m=0$ we have
	\begin{align}\label{eq:elliptic2}
		\exp\int_0^{1/2-\varepsilon} E_j &\sim_{\varepsilon \to 0} |\det (I_{n-n_j} -T_j)|\, (2\pi\varepsilon)^{-\frac n {\nu_j} + n_j}\, \nu_j^{-n_j}.
	\end{align}
\end{lem}

\begin{proof}
	Using that $\rho(c_j^{\nu_j})=\rho(u)=I_n$, it is easy to see by substituting $k$ for $\nu_j-k$ in \cref{eq:Ei} that $E_j(\xi)$ is real for $\xi\in\R$. Therefore, $\exp\int_0^{1/2-\varepsilon}E_j$ is positive for sufficiently small $\varepsilon>0$. The statement now follows by the same computations as in \cite[proof of Lemma 6]{fried1986fuchsian}.
\end{proof}

\begin{proof}[Proof of \cref{theo:Ruelle0}]
First we assume $m=0$. Then $\rho(u) = I_n$, and \cref{theo:Ruelle0}~\cref{item1} reads
$$R\left(\frac s {2\pi}; \rho\right) \sim_{s\to 0} (-1)^{\mult(0;\Delta_{\tau_m,\rho}^\sharp)+\sum_{j=1}^r n_j}\frac{s^{n(2g-2+r) - \sum_{j=1}^r n_j}}{\prod_{j=1}^r|\det (I_{n-n_j} - T_j)|\,\nu_j^{-n_j}} + o(s).$$
Now combining the second statement of \cref{lem1} with \cref{lem3} yields
\begin{multline*}
	\eta\left(1+\frac s {2\pi};\rho\right) \sim_{s\to0} (-s)^{-\frac{\dim(V_\rho)\Vol(X)}{2\pi}}\\
	\times\prod_{j=1}^r |\det(I_{n-n_j} -T_j)|\,(-s)^{-\dim(V_\rho)/\nu_j + n_j}\, \nu_j^{-n_j} + o(s).
\end{multline*}
By the Gauss--Bonnet Theorem (see \cref{eq:GaussBonnet}) this can be written as
$$ (-s)^{-n(2g-2+r)+\sum_{j=1}^rn_j}\prod_{j=1}^r |\det(I_{n-n_j} -T_j)|\, \nu_j^{-n_j} + o(s), $$
so the claim follows with \cref{prop:RuelleVsEtaAtZero}.

Now if $m\neq 0$, we deduce from \cref{prop:RuelleVsEtaAtZero}, the first statement in \cref{lem1} and \cref{lem2}, using $\rho(u)= e^{-i\pi m} I_n$:
$$ R(0;\rho) = (-1)^{\mult(0;\Delta_{\tau_m,\rho}^\sharp)}\frac{\det(I_n-\rho(u)^\varepsilon)^{2g-2+r}}{\prod_{j=1}^r\det(I_n-\rho(c_j)^\varepsilon)} $$
with $\varepsilon=\sign(m)$. For $\varepsilon=1$ we immediately obtain \cref{theo:Ruelle0}~\cref{item2}, and for $\varepsilon=-1$ we can use
$$\det(I_n-\rho(u)^{-1})=(-1)^n\det(\rho(u))^{-1}\det(I_n-\rho(u))$$
and the analogous identity for $c_j$ together with the fact that $\prod_{j=1}^r \det(\rho(c_j))=\det(\rho(u))^{2g-2+r}$ (see the proof of \cref{lem:CenterActionIrredPi1X1}).
\end{proof}

Now, \cref{theo:mainFried} is a consequence of \cref{prop:torsion} and \cref{theo:Ruelle0}.

\providecommand{\bysame}{\leavevmode\hbox to3em{\hrulefill}\thinspace}
\providecommand{\href}[2]{#2}

\contact

\end{document}